\documentclass[10pt,twoside,a4paper]{amsart}
\usepackage{amssymb,amsmath,amsfonts,amsthm,graphicx,amscd}
\usepackage{enumerate,mathrsfs,longtable,bm,float,xcolor}
\usepackage[dvips]{epsfig}
\usepackage{amsmath,amsfonts,amssymb,amsthm,bm,float,makecell,booktabs,MnSymbol}

\usepackage{enumerate}
\usepackage{amscd}
\usepackage{tikz}
\usetikzlibrary{positioning}
\usepackage[all,cmtip,line]{xy}

\DeclareMathAlphabet{\mathpzc}{OT1}{pzc}{m}{it}

\numberwithin{equation}{section}

\theoremstyle{plain}

\newtheorem{lem}{Lemma}[section]

\newtheorem{thm}[lem]{Theorem}
\newtheorem{prop}[lem]{Proposition}
\newtheorem{obs}[lem]{Observation}
\newtheorem{cor}[lem]{Corollary}
\theoremstyle{definition}
\newtheorem{exa}[lem]{Example}
\newtheorem{rem}[lem]{Remark}
\newtheorem{defn}[lem]{Definition}
\newtheorem{definition}[lem]{Definition}

\begin{document}
	
	\baselineskip 13truept
	
	\title{On $S$-prime and $S$-primary elements in multiplicative lattices}
	
	\author{Sachin Sarode*, Chetan Patil** and Vinayak Joshi***}
	\address{\rm *Department of Mathematics, Shri Muktanand College\\ Gangapur, Dist. Chh. Sambhajinagar - 431 109, India.} \email{sarodemaths@gmail.com}

	\address{\rm ***School of Technology Management and Engineering, SVKM NMIMS Global University,
		Dhule-424 001, India.}
	\email{patilcs19@gmail.com}

	\address{\rm **Department of Mathematics, Savitribai Phule Pune University, Pune-411 007, India.}
	\email{vvjoshi@unipune.ac.in \\
		vinayakjoshi111@yahoo.com }


	\begin{abstract} 
		  In this paper, we  study  $S$-prime elements and $S$-primary elements within the framework of multiplicative lattices. Furthermore, we define and explore weakly $S$-prime elements and weakly $S$-primary elements, which generalize weakly prime elements and weakly primary elements in multiplicative lattices respectively. 
		 We show that the weakly $S$-prime ideals (weakly $S$-primary ideals) of a commutative ring $R$ with $1$ correspond precisely to the weakly $S_L$-prime elements (weakly $S$-primary elements) of the ideal lattice $Id(R)$ of $R$, where $S_L = \{(s) \mid s \in S\}$.
	\end{abstract}

	\maketitle
	
	\vskip 5truept
	
	\vskip 5truept
	

	\baselineskip 14truept 
	\section{Introduction}

 The concepts of prime and primary ideals are fundamental notions in commutative ring theory, and their generalizations have been extensively studied \cite{ABT, AB, AB1, AS, CK, HM, M, MG, V}. The ideal theory in rings with identity can be explored via multiplicative lattice theory, where prime and primary elements have also been generalized \cite{CYT, CUU, CUU1, JTY, MB, SPJ1, SPJ2}.
		
Hamed and Malek \cite{HM} introduced $S$‑prime ideals, extending classical prime ideals. An ideal $P$ disjoint from a multiplicative closed subset $S$  of a ring $R$ with identity is \textit{$S$‑prime} if there exists $ s\in S$ such that for all $a,b\in R$, $ab\in P$ implies $sa\in P$ or $sb\in P$. 
Independently, Massaoud \cite{M} and Visweswaran \cite{V} introduced $S$‑primary ideals, later extended to weakly $S$‑primary ideals by Celikel–Khashan \cite{CK} and Massaoud–Gouaid \cite{MG}. Similarly, Almahdi et al. \cite{ABT} and Mahdou et al. \cite{MMZ} defined weakly $S$‑prime ideals. Analogously,  Sarode, Patil and Joshi \cite{SPJ1} defined $S$‑prime elements in $V$‑lattices.  Now, we introduce and study weakly $S$‑prime elements in $V$‑lattices and weakly $S$‑primary elements in $c$‑lattices.
We prove that weakly $S$‑prime (weakly $S$‑primary) ideals of a commutative ring $R$ with $1$ correspond to weakly $S_L$‑prime (weakly $S_L$‑primary)  elements of the $c$‑lattice $\operatorname{Id} (R)$, where $S_L=\{ (s)\mid s\in S\}$. 
		 
We begin with some basics definitions. 

\begin{defn}[Sarode, Patil and Joshi \cite{SPJ1}] 
	A complete lattice $L$	is said to be a \textit{$V$-lattice}, if there exists a binary	operation $`` \cdot "$ called multiplication on $L$ satisfying the following conditions for $a, b, c \in L$: 
	\begin{enumerate} \item $ a \cdot b = b \cdot a $, 
		
	\item $a \cdot (b \cdot c) = (a \cdot b) \cdot c $, 
				
	\item $a \leq b$ implies $a \cdot c \leq b \cdot c $,
				
	\item $a \cdot b \leq a \wedge b $, 
				
	\item $1 \cdot a = a \cdot 1 = a$.
	\end{enumerate} 
\end{defn}
		
	A \textit{multiplicative lattice} is a complete lattice with a commutative, associative multiplication distributing over infinite joins and identity $1$.
	An element $c$ is called \textit{compact} if it satisfies $c\leq \bigvee _{\alpha }a_{\alpha }$ only when $c\leq \bigvee _{i=1}^na_{\alpha _i}$ for some finite subset. Let $L_*$ denotes the set of all compact elements of $L$. A multiplicative lattice $L$ is \textit{compactly generated} if every element is a join of compact elements. If all elements are compact, $L$ is \textit{compact}; if $1$ is compact, $L$ is \textit{$1$‑compact}. A \textit{$c$‑lattice} is a multiplicative lattice that is $1$‑compact, compactly generated, and closed under finite products of compact elements. A subset $S\subseteq L_*$ is multiplicatively closed if $a\cdot b\in S$ whenever $a,b\in S$. For $a,b\in L$, $(a:b)=\bigvee \{ x\mid x\cdot b\leq a\}$, and the radical of $a$ is $\sqrt{a}=\bigvee \{ x\in L_*\mid x^n\leq a\mathrm{\  for\  some\  }n\in \mathbb{Z^{\mathnormal{+}}}\}$; see \cite{AAJ, A, SJ, WD} .
		
	Note that a multiplicative  lattice  is a
	$V$-lattice  but  every $V$-lattice  need not be a multiplicative lattice. 
		
	Throughout, $S$ is a multiplicatively closed subset of a $V$-lattice $L$ such that $0\nin S$ and $1\in S$. 
		
	\begin{defn}
	An element $p \neq 1$ of a $V$-lattice $L$  is \textit{prime} element if $a	\cdot b \leq p$ implies $a \leq p$ or $b \leq p$ for all $a, ~b \in L$.  		
	\end{defn}
	
	\begin{defn} 
	An element $a$ of $V$-lattice $L$ is said to be \textit{zero-divisor} if there exists $b\in L\setminus \{0\}$ such that $a\cdot b=0$. The set of all zero-divisors of $L$ is denoted by $Z(L)$. 
	\end{defn}
		
	For further details on multiplicative lattices and S‑prime elements, see \cite{AAJ, A, D, JJJ, JS, S, SJ, WD}. Hereafter, we shall denote $a \cdot b$ simply by $ab$.

	\section{$S$-prime element}
		
	The following definition can be found in \cite{SPJ1}.
		
	\begin{defn}[Sarode, Patil and Joshi \cite{SPJ1}]
	Let $S$ be a multiplicatively closed subset of a $V$-lattice $L$. A proper element $p \in L$ satisfying $t \not\leq p$ for all $t \in S$ is called an \textit{$S$-prime} element of $L$ if there exists $s \in S$ 
	such that for all $a, b \in L$,   if $ab \leq p $, then  $sa \leq p \; \text{or} \; sb \leq p$.
	The set of all $S$-prime elements of $L$ is denoted by $Spec_s(L)$. If $S = \{1\}$, then $Spec_s(L)$ is denoted by $Spec(L)$, the set of all prime elements.
	\end{defn}

	\begin{rem} [\cite{D, SJ, WD}]
	Let $a,b $ be any elements of a $c$-lattice $L$. Then  $x b \leq a
	\Leftrightarrow x \leq (a:b)$. 
	\end{rem}

	\begin{lem} [Sarode, Patil and Joshi \cite{SPJ1}]\label{L2.1}
	Let $S$ be a multiplicatively closed subset of a multiplicative lattice $L$. If $p$ is proper element of $L$ such that $t \nleq p$ for all $t \in S$. Then $p$ is an $S$-prime element of $L$ if and only if $(p : s)$ is a prime element of $L$ for some $s \in S$. 
	\end{lem}

	\begin{defn} [\cite{D}]
	Let $a$ be any element of a multiplicative lattice $L$, then $L/a =\{ x \in L \;|\; x \geq a\}$ is  a multiplicative  lattice with the  multiplication $x \circ y = x  y \vee a$.  Let $p$ be a proper element of a multiplicative lattice $L$. 
	For any $a \in L$, write $\bar{a} = a \vee p$. The \textit{zero-divisor set} of $L/p$ is defined as	$Z(L/p) = \{ \bar{a} \in L/p \mid \bar{a} \neq \bar{p},\ 	\exists\, \bar{b} \neq \bar{p} \text{ in } L/p \text{ such that } \bar{a} \circ \bar{b} = \bar{p} \},$		where $\bar{p}$ plays the role of the zero element in $L/p$.
	\end{defn}

	\begin{prop}\label{L2.5} 
	Let $S$ be a multiplicatively closed subset of a multiplicative lattice $L$ and $p$ is a proper element of $L$ such that $t \nleq p$ for all  $t\in S$. Consider the set $\bar{S} = \{s \vee p \;|\; s \in S \}$. If $Z(L/p) \cap \bar{S} = \emptyset$, then $p$ is an $S$-prime element of $L$ if and only if  $p$ is  a prime element of $L$.
	\end{prop}

	\begin{proof}
	If $p$ is  a prime element of $L$, then clearly $p$ is  an $S$-prime element of $L$, as $t \nleq p$ for all  $t\in S$. Conversely, suppose that $p$ is an $S$-prime element of $L$.  We claim that $(p:s) = p$ for all $s \in S$. Clearly, $p\leq (p:s)$ for all $s \in S$. Now, we show that $(p:s)\leq p$. Let $x \in L$ such that $x \leq (p:s)$. This gives $sx \leq p$. Consider $(s \vee p), ~ (x \vee p) \in L/p$, we get  $(s \vee p) \circ (x \vee p) =(s \vee p) \cdot (x \vee p) \vee p= p $ in $L/p$. As  $t \nleq p$ for all  $t\in S$, we get $s \vee p \neq p$ in $L/p$. Also, since $Z(L/p) \cap \bar{S} = \emptyset$, we get $s \vee p \not\in Z(L/p) $. Hence $x \vee p = p$ in $L/p$. Thus $x \leq p$. Since every  element below $ (p:s)$ is also below $ p$, in particular, we can take $x=(p:s)$. Thus, we get $(p:s) \leq p$. Therefore $(p:s) =p$, by Lemma \ref{L2.1}, $p$ is  a prime element of $L$. 
	\end{proof}

	\begin{prop}
	Let $S$ be a multiplicatively closed subset of a $c$-lattice $L$ and $i, p $ be any element of $L$ such that  $i \in L_*$ and $i \leq p$. Then $p$ is an $S$-prime element of $L$ if and only if $p$ is an $\bar{S}$-prime element of $L/i$, where $\bar{S} = \{s \vee i \;|\; s\in S\}$.
	\end{prop}
	\begin{proof}  First, we will show that $\bar{S}$ is a multiplicatively closed subset of $L/i$. Since $i$ is compact and every $s \in S$ is compact, $s \vee i$ is compact for every $s \in S$. Therefore $\bar{S} \subseteq  (L/i)_*$. Let $s_1 \vee i, ~ s_2 \vee i \in \bar{S} $. therefore $(s_1 \vee i) \circ (s_2 \vee i) = s_1 s_2 \vee s_1 i \vee s_2 i \vee i^2 \vee i = s_1 s_2 \vee  i $. This gives $(s_1 \vee i) \circ (s_2 \vee i) \in \bar{S}$. Hence $\bar{S}$ is multiplicatively closed subset of $L/i$. 
			
	We claim that $t' \nleq p$ for all $t' \in \bar{S}$. Suppose on the contrary that $t'_1 \leq p$ for some $t'_1 \in \bar{S}$. Therefore $t_1 \vee i \leq p$, where $t'_1 = t_1 \vee i$ and $t_1 \in S$. This gives $t_1 \leq p$,  a contradiction to  $t \not\leq p$ for all $t \in S$, as $p$ is an $S$-prime element. Hence  $t' \nleq p$ for all $t' \in \bar{S}$.  

	Now, we prove that $p$ is an $\bar{S}$-prime element of $L/i$. Let $x, y \in L/i$ with $x \circ y \leq p$. Therefore    $x \circ y  = xy \vee i \leq p$. Hence $xy \leq p$. As $p$ is an $S$-prime element of $L$, there exists $s \in S$ such that $xs \leq p$ or $ys \leq p$.  Since  $i \leq p$, we get $\bar{s}x=(s \vee i) x = sx \vee ix \leq p$ or   $\bar{s}y=(s \vee i) y = sy \vee iy \leq p$, where $\bar{s}=s\vee i$.  Therefore $p$ is an $\bar{S}$-prime element of $L/i$.
			
	Conversely, suppose that  $p$ is an $\bar{S}$-prime element of $L/i$.  Clearly, $t \nleq p$ for all $t \in S$.  Let $c, d \in L$ such that $ cd \leq p$. Since $i \leq p$, we get $(c \vee i ) (d \vee i) = cd \vee ci \vee di \vee i^2 \leq p$. As $(c \vee i), (d \vee i) \in L/i$ and $p$ is an $\bar{S}$-prime element of $L/i$ there exists $\bar{s}=s \vee i \in \bar{S}$ such that $(s \vee i) (c \vee i) \leq p$ or $(s \vee i) (d \vee i) \leq p$. This gives $sc \vee si \vee ci \vee i^2 \leq p$ or $sd \vee si \vee di \vee i^2 \leq p$.   Therefore $sc \leq p $ or $sd \leq p$. Hence $p$ is an $S$-prime element.
	\end{proof}

	\begin{lem}
	Let $S$ be a multiplicatively closed subset of a $V$-lattice $L$.  Let $q$ be a proper element of $L$ such that $q \in S$. If $p$ is an $S$-prime element of $L$, then $p  q$ is an $S$-prime element of $L$.
	\end{lem}
	\begin{proof} 
	First, we show that $t \nleq p  q$ for all $t \in S$. Suppose on the contrary that there exists an element $t \in S$ such that $t \leq p  q $. Since $ p  q \leq p$, we will get a contradiction to $t \nleq p$ for all $t \in S$. Hence $t \nleq p q$ for all $t \in S$. Since $1 \in S$ and $t \nleq p  q$ for all $t \in S$, we get $pq$ is a proper element of $L$. Let $a, b \in L$ such that $a  b \leq p  q$. Since $p  q \leq p$, we get $a  b \leq p$. As $p$ is an $S$-prime element of $L$, there exists an element $s \in S$ such that $s  a \leq p$  or $s  b \leq p$.  Therefore $q  s  a \leq p  q$  or $q  s  b \leq p  q$, where $q  s \in S$. Hence, $p  q$ is an $S$-prime element of $L$.  
	\end{proof}

	\begin{lem}\label{L2.8}
	Let  $S$ be a multiplicatively closed subset of a $c$-lattice $L$.  If $p$ is an $S$-prime element of $L$ and $q$ be a proper element of $L$ such that $q \leq p$, then there exists an element $s \in S$ such that $s  \sqrt{q} \leq p$.
	\end{lem}
	\begin{proof} Since $p$ is an $S$-prime element of $L$, by Lemma \ref{L2.1}, $(p : s)$ is a prime element of $L$ for some $s \in S$.	Let $z$ be any compact element of $L$ such that $z \leq \sqrt{q}$. Therefore $z^n \leq q \leq p \leq (p : s)$ for some $n \in \mathbb{N}$. By primeness of $( p : s)$, $z \leq (p : s)$. Hence  $\sqrt{q} \leq (p :s)$ for $s \in S$, that is,  $s \sqrt{q} \leq p$.      
	\end{proof}

	\begin{lem}
	Let  $S$ be a multiplicatively closed subset of a $c$-lattice $L$.   If $p_1,\cdots, p_n$ are  $S$-prime elements of $L$, then there exists an element $s \in S$ such that $s \cdot \sqrt{p_1 \wedge \cdots \wedge p_n} \leq p_1 \wedge \cdots\wedge p_n$.
	\end{lem}
		
	\begin{proof}
	Suppose $p_1, \cdots, p_n$ are  $S$-prime elements of $L$. By Lemma \ref{L2.8}, there exists an elements $s_i \in S$ such that $s_i  \sqrt{p_i} \leq p_i$. We have $\sqrt{p_1 \wedge \cdots \wedge p_n} = \sqrt{p_1} \wedge \cdots \wedge \sqrt{p_n}$. Let $s = s_1 s_2 \cdots s_n \in S$. Then $s  \sqrt{p_1 \wedge \cdots \wedge p_n}  \leq s  \sqrt{p_i} \leq p_i$ for every $i$. Hence $ s  (\sqrt{p_1 \wedge \cdots  \wedge p_n}) = s  (\sqrt{p_1} \wedge \cdots  \wedge \sqrt{p_n})  \leq p_1 \wedge \cdots \wedge p_n$. 
	\end{proof}
		
		

\begin{definition}[\cite{Birkhoff1967},\cite{WD}]
	Let $L_1, L_2, \dots, L_n$ be multiplicative lattices. Their \emph{finite direct product}, denoted by
	\[
	L_1 \times L_2 \times \cdots \times L_n,
	\]
	is the set of all ordered $n$-tuples
	\[
	(x_1, x_2, \dots, x_n),
	\]
	where $x_i \in L_i$ for each $i = 1,2,\dots,n$.
	
	The structure of a multiplicative lattice on $L_1 \times \cdots \times L_n$ is given by defining the order relation, lattice operations, and multiplication component-wise. Specifically, for any two elements $(x_1,\dots,x_n)$ and $(y_1,\dots,y_n)$, set
	\[
	(x_1,\dots,x_n) \leq (y_1,\dots,y_n) \quad \text{if and only if} \quad x_i \leq y_i \text{ for all } i,
	\]
	and define
	\[
	(x_1,\dots,x_n) \vee (y_1,\dots,y_n) = (x_1 \vee y_1,\dots,x_n \vee y_n),
	\]
	\[
	(x_1,\dots,x_n) \wedge (y_1,\dots,y_n) = (x_1 \wedge y_1,\dots,x_n \wedge y_n),
	\]
	\[
	(x_1,\dots,x_n)(y_1,\dots,y_n) = (x_1 y_1,\dots,x_n y_n).
	\]
	
	If each lattice $L_i$ possesses a multiplicative identity element $1_i$, then the element $(1_1,1_2,\dots,1_n)$ serves as the multiplicative identity of $L_1 \times \cdots \times L_n$.
	
	With these operations, $L_1 \times \cdots \times L_n$ becomes a multiplicative lattice.
\end{definition}

\begin{prop}
	Let $L_1$ and $L_2$ be two $c$-lattices and $S_1$, $S_2$ be the multiplicatively closed subsets of $L_1$, $L_2$ respectively. Let $S=S_1\times S_2$ be multiplicatively closed subset of  $L=L_1\times L_2$. Then the following statements holds:
	\begin{enumerate}
		\item\label{P2.11.1} $p_1$ is an $S_1$-prime element of $L_1$ if and only if $(p_1,1)$ is an $S$-prime element of $L$.
		\item\label{P2.11.2} $p_2$ is an $S_2$-prime element of $L_1$ if and only if $(1,p_2)$ is an $S$-prime element of $L$.
	\end{enumerate}
\end{prop}    
\begin{proof} 
	(\ref{P2.11.1})	Note that , $t_1\nleq p_1$ for all $t_1\in S_1$ if and only if $(t_1,t_2)\nleq (p_1,1)$ for all $(t_1,t_2) \in S$. Also $t_2 \nleq p_2$ for all $t_2 \in S_2$ if and only if $(t_1,t_2) \nleq (1,p_2)$ for all $(t_1,t_2) \in S$. Let $(a_1,b_1), (a_2,b_2)\in L_1\times L_2$ such that $(a_1,b_1)(a_2,b_2)\leq (p_1,1)$. Then $a_1a_2\leq p_1$. Since $p_1$ is a $S_1$-prime element of $L_1$, there exists $s_1\in S_1$ such that $s_1a_1\leq p_1$ or $s_1a_2\leq p_1$. Hence $(s_1,1)(a_1,b_1)=(s_1a_1,b_1)\leq (p_1,1)$ or $(s_1,1)(a_2,b_2)=(s_1a_2,b_2)\leq (p_1,1)$. Thus, $(p_1,1)$ is an $S$-prime element of $L$.
	
	Conversely, suppose that $(p_1,1)$ is an $S$-prime element of $L$. Let $a,b\in L_1$ such that $ab\leq p_1$. Therefore $(a,1)(b,1)=(ab,1)\leq (p_1,1)$. Since $(p_1,1)$ is an $S$-prime element of $L$, there exists an element $(s_1,s_2)\in S$ such that $(s_1a,s_2)=(s_1,s_2)(a,1)\leq (p_1,1)$ or $(s_1b,s_2)=(s_1,s_2)(b,1)\leq (p_1,1)$. This implies $s_1a\leq p_1$ or $s_1b\leq p_1$. Thus $p_1$ is an $S_1$-prime element of $L_1$.
	
	Proof of 	(\ref{P2.11.2}) is same as that of  (\ref{P2.11.1}).
\end{proof}

\section{$S$-primary element}

\begin{defn}[\cite{SPJ2}]
	Let $S$ be a multiplicatively closed subset of a $c$-lattice $L$. A proper element $q$ of $L$ such that $t\nleq q$ for all $t\in S$ is called an S-primary element if there exists an element $s\in S$ such that for all $c,d\in L$ with $cd\leq q$, implies $sc\leq q$ or $sd\leq \sqrt q$.	
	
	Note that if $L$ is a $c$-lattice and $S=\{1\}$, then the primary elements of $L$ and the $S$-primary elements of $L$ coincide.
\end{defn}


\begin{prop}
	Let $S$ be a multiplicatively closed subset of a $c$-lattice $L$. Let $q$ be a primary element of $L$ such that $t \nleq q$ for all $t \in S$. Then for any $s \in S$, $sq$ is an $S$-primary element of $L$.
\end{prop}

\begin{proof}
	Since $q$ is a primary element of $L$, $q$ is a proper element of $L$. Let any $s \in S$ and consider $sq$. As $sq \leq q$ and $t \nleq q$ for all $t \in S$, we have $t \nleq sq$ for all $t \in S$. Let $ c, d \in L$ such that $cd \leq sq \leq q$. As $q$ is a primary element of $L$, this gives $c \leq q$ or $d \leq \sqrt{q}$. Therefore $sc \leq sq$ or $sd \leq \sqrt{sq}$. This shows that $sq$ is an $S$-primary element of $L$.
\end{proof}

\begin{prop}
	Let $S$ be a multiplicatively closed subset of a $c$-lattice $L$.	If $q$ is an $S$-primary element of $L$, then $\sqrt q$ is an $S$-prime element of $L$.
\end{prop}

\begin{proof} 
	Let $c,d \in L$ such that $cd \leq \sqrt q$. Then $c^n d^n \leq q$ for some positive integer $n$. Since $q$ is an $S$-primary element of $L$,  there exists $s\in S$ such that $s c^n\leq q$ or $sd^n\leq \sqrt q$. Hence, $(sc)^n\leq q$ or $(sd)^n\leq \sqrt q$. Therefore, $sc\leq \sqrt q$ or $sd\leq \sqrt q$. Thus, $\sqrt q$ is an $S$-prime element of $L$.
\end{proof}	

\begin{prop}\label{P3.4} 
	Let $S$ be a multiplicatively closed subset of a $c$-lattice $L$. If $q$ is a proper element of $L$ such that $t\nleq q$ for all $t\in S$, then the following statements are equivalent:
	\begin{enumerate}
		\item $q$ is an $S$-primary element of $L$.
		
		\item $(q:s)$ is a primary element of $L$ for some $s\in S$.
	\end{enumerate}
\end{prop}

\begin{proof} 
	$(1)\implies (2)$: Suppose $q$ is an $S$-primary element of $L$. Therefore, for all $c,d\in L$ with $cd\leq q$, there exists $s\in S$ such that $sc\leq q$ or $sd\leq \sqrt q$. For this element $s$, we claim that $(q:s)$ is a primary element of $L$.
	
	Let $a,b \in L$ such that $ab \leq (q:s)$. Then $sab \leq q$. Since $q$ an $S$-primary element of $L$, we have $s^2a \leq q$ or $sb\leq \sqrt q$. If $s^2a\leq q$, then $s^3\leq \sqrt q$ or $sa\leq q$. Since $t\nleq q$ for all $t\in S$, we could not have $s^3\leq \sqrt{q}$. Therefore  $s^3\nleq \sqrt q$. Hence, $sa\leq q$. This gives $a \leq (q:s)$. If $sb \leq \sqrt q$, then $s^nb^n\leq q$ for some positive integer $n$. Therefore $s^{n+1}\leq \sqrt q$ or $sb^n\leq q$. Since $t\nleq q$ for all $t\in S$, we could not have $s^{n+1} \leq \sqrt q$. Therefore $s^{n+1}\nleq \sqrt q$. Hence, $sb^n \leq q$. This gives $b \leq \sqrt{(q:s)}$.  Thus, $a \leq (q:s)$ or $b\leq \sqrt{(q:s)}$. This proves that $(q:s)$ is a primary element of $L$.
	
	$(2)\implies (1)$: Suppose $(q:s)$ is a primary element of $L$ for some $s\in S$. We have to show that $q$ is an $S$-primary element of $L$. By assumption $q$ is a proper element of $L$ such that $t\nleq q$ for all $t\in S$.   Now, let $x,y\in L$ such that $xy\leq q$. As $q \leq (q:s)$, we get  $xy\leq (q:s)$. Since $(q:s)$ is a primary element of $L$, $x \leq (q:s)$ or $y\leq \sqrt {(q:s)}$. This implies $sx\leq q$ or $sy^n\leq q$ for some positive integer $n$. Hence, $sx\leq q$ or $(sy)^n\leq q$ for some positive integer $n$. Therefore, $sx\leq q$ or $sy\leq \sqrt q$. This proves that $q$ is an $S$-primary element of $L$.
\end{proof}

\begin{prop}
	Let $S$ be a multiplicatively closed subset of a $c$-lattice $L$ and $q$ is a proper element of $L$ such that $t \nleq q$ for all  $t\in S$. Consider the set $\bar{S} = \{s \vee q \;|\; s \in S \}$. If $Z(L/q) \cap \bar{S} = \emptyset$, then $q$ is an $S$-primary element of $L$ if and only if  $q$ is  a primary element of $L$.
\end{prop}

\begin{proof} proof is similar to the proof of Proposition \ref{L2.5}.
\end{proof}	

\begin{prop}\label{P3.6} 
	Let $S$ be a multiplicatively closed subset of a $c$-lattice $L$ and $i, q $ be any element of $L$ such that  $i \in L_*$ and $i \leq q$. Then $q$ is an $S$-primary element of $L$ if and only if $q$ is an $\bar{S}$-primary element of $L/i$, where $\bar{S} = \{s \vee i \;|\; s\in S\}$.
\end{prop}

\begin{proof}  Firstly, we will show that $\bar{S}$ is a multiplicatively closed subset of $L/i$. Since $i$ is compact and every $s \in S$ is compact, $s \vee i$ is compact for every $s \in S$. Therefore $\bar{S} \subseteq  (L/i)_*$. Let $s_1 \vee i, ~ s_2 \vee i \in \bar{S} $. therefore $(s_1 \vee i) \circ (s_2 \vee i) = s_1 s_2 \vee s_1 i \vee s_2 i \vee i^2 \vee i = s_1 s_2 \vee  i $. This gives $(s_1 \vee i) \circ (s_2 \vee i) \in \bar{S}$. Hence $\bar{S}$ is multiplicatively closed subset of $L/i$. We claim that $t' \nleq q$ for all $t' \in \bar{S}$. Suppose on the contrary that $t'_1 \leq q$ for some $t'_1 \in \bar{S}$. Therefore $t_1 \vee i \leq q$, where $t'_1 = t_1 \vee i$ and $t_1 \in S$. This gives $t_1 \leq q$, as $q$ is an $S$-primary element we will get a contradiction to  $t \leq q$ for all $t \in S$. Hence  $t' \nleq q$ for all $t' \in \bar{S}$.  Let $x, y \in L/i$ with $x \circ y \leq q$. Therefore    $x \circ y  = xy \vee i \leq q$. Since $i \leq q$, we get $xy \leq q$. As $q$ is an $S$-primary element of $L$, there exists $s \in S$ such that $xs \leq q$ or $ys \leq \sqrt{q}$.  Since $sx \leq q$ and $i \leq q$, we get $(s \vee i) x = sx \vee ix \leq q$.  Also as  $ys \leq \sqrt{q}$ and $i \leq q \leq \sqrt{q}$, we get $(s \vee i) y = sy \vee iy \leq \sqrt{q}$.  Therefore $q$ is an $\bar{S}$-primary element of $L/i$.
	
Conversely, suppose that  $q$ is an $\bar{S}$-primary element of $L/i$. So Clearly, $t \nleq q$ for all $t \in S$.  Let $c, d \in L$ such that $ cd \leq q$. Since $i \leq q$, we get $(c \vee i ) (d \vee i) = cd \vee ci \vee di \vee i^2 \leq q$. As $(c \vee i), (d \vee i) \in L/i$ and $q$ is an $\bar{S}$-primary element of $L/i$.  there exists $s \vee i \in \bar{S}$ such that $(s \vee i) (c \vee i) \leq q$ or $(s \vee i) (d \vee i) \leq \sqrt{q}$. This gives $sc \vee si \vee ci \vee i^2 \leq q$ or $sd \vee si \vee di \vee i^2 \leq \sqrt{q}$. As $i \leq q \leq \sqrt{q}$, we get $ sc \vee si \vee ci \vee i^2 \vee i = sc \vee i  \leq q$ or $ sd \vee si \vee di \vee i^2 \vee i  = sd \vee i \leq \sqrt{q} $. Therefore $sc \leq sc \vee i \leq q $ or $sd \leq sd \vee i \leq \sqrt{q}$. Hence $q$ is an $S$-primary element of $L$.
\end{proof}

\begin{prop}
	Let $S$ be a multiplicatively closed subset of a $c$-lattice $L$ and $q\in L$ such that $t\nleq q$ for all $t\in S$. Then $q$ is an $S$-primary element of $L$ if and only if there exists $s\in S$ such that for all $i_1,i_2,\dots,i_n\in L$ with $i_1i_2\cdots i_n\leq q $, $si_j\leq q$ or $si_k\leq \sqrt{q}$ for some $j,k\in \{1,2,\dots,n\}$.
\end{prop}	
\begin{proof} 
	Suppose $q$ is an $S$-primary element of $L$. Let $i_1,i_2,\dots,i_n\in L$ with $i_1i_2\cdots i_n\leq q $. We have to prove that $si_j\leq q$ or $si_k\leq \sqrt{q}$ for some $j,k\in \{1,2,\dots,n\}$. We prove this result by induction on $n$, the number of $i_j$. For $n=1$, the result is obvious. For $n=2$, the result follows from the definition of an $S$-primary element. Let $n\geq 3$ and assume that the result holds for number $(n-1)$. Then by the definition of an $S$-primary element, there exists $s\in S$ such that $s(i_1 i_2\cdots i_{n-1})\leq q$ or $si_n\leq \sqrt{q}$. Therefore $si_n\leq \sqrt{q}$ or $s^2 i_j \leq q$ or $s i_k \leq \sqrt{q}$ for some $j,k \in \{1,2,\cdots ,n-1\}$. If $s^2 i_j \leq q$, we get $s i_j\leq q $ or $ s^3\leq \sqrt{q}$. Since $t \nleq q$ for all $t \in S$, we could not have $s^3 \leq \sqrt{q}$. Hence  $s i_j \leq q$.
\end{proof}

We defined $S$-stationary property for sequence of elements of $V$-lattice $L$ in \cite{SPJ2}.

\begin{defn}[\cite{SPJ2}]
	Let $S$ be a multiplicatively closed subset of $V$-lattice $L$. An ascending sequence of elements $i_1 \leq i_2 \leq i_3 \leq \cdots \leq i_m  \leq \cdots $ is called \textit{$S$-stationary} if there exist $s \in S$ and $ n \in \mathbb{Z}^{+}$ such that $s  i_m \leq i_n$ for all $m \geq n$. 
	
	Let $S$ be a multiplicatively closed subset of $V$-lattice $L$, then  $L$ is said to be satisfies $S$-stationary property, if every ascending chain of elements is $S$-stationary.	
\end{defn}

\begin{prop}
	Let $S$ be a multiplicatively closed subset of a $c$-lattice $L$ and $q\in L$ such that $t\nleq q$ for all $t\in S$. Consider the following statements:
	\begin{enumerate}
		\item $q$ is an $S$-primary element of $L$.
		\item $(q:s)$ is a primary element of $L$ for some $s\in S$.
		\item The ascending chain of elements $(q:sl) \leq (q:sl^2) \leq (q:sl^3) \leq \cdots $ is stationary for some $s \in S$ and for all $l \in L$.
		\item The ascending chain of elements $(q:l) \leq (q:l^2) \leq (q:l^3) \leq \cdots $ is $S$-stationary for all $l \in L$.
	\end{enumerate}
	Then $(1)\implies (2)\implies (3) \implies (4)$.
\end{prop}	

\begin{proof} 
	
	$(1) \implies (2)$ Follows from the Proposition \ref{P3.4}.
	
	$(2)\implies (3)$ Let $l \in L$. Suppose that $l \nleq \sqrt{(q:s)}$. We claim that $(q:s) = (q:sl^n)$. Let $x \in L_*$ such that $x \leq (q:sl^n)$. Then $sxl^n\leq q$. Hence $xl^n\leq (q:s)$. As $(q:s)$ is primary, $x\leq (q:s)$ or $l^n\leq \sqrt{(q:s)}$. Therefore $x\leq (q:s)$ or $l \leq \sqrt{(q:s)}$. But $l \nleq \sqrt{(q:s)}$. Hence $x\leq (q:s)$. Thus every compact element $\leq (q:sl^n)$ is $\leq (q:s) $ This gives $(q:sl^n)\leq (q:s)$. Let $y \in L_{*}$ such that $y\leq (q:s)$. Then $ys\leq q$. Therefore $ysl^n\leq q$. This implies $y\leq (q:sl^n)$. So every compact element $\leq (q:s) $  is  $\leq (q:sl^n)$. This gives $(q:s) \leq (q:sl^n)$. Thus, $(q:s) = (q:sl^n)$. 
	\par Suppose that $l \leq \sqrt{(q:s)}$. Therefore $sl^k\leq q$ for some positive integer $k$. Hence for all $j\geq k$, $sl^j\leq q$. Thus $(q:sl^j)=1$. This proves the ascending chain of elements $(q:sl) \leq (q:sl^2) \leq (q:sl^3) \leq \cdots $ is stationary for some $s \in S$ and for all $l \in L$.
	
	$(3)\implies (4)$ Let $l \in L$. Then there exists positive integer $n$ such that for all $j \geq n$, $(q:sl^j)=(q:sl^n)$. We claim that $s(q:l^j)\leq (q:l^n)$ for all $j\geq n$. Let $x\leq (q:l^j)$. Then $xl^j\leq q$. Therefore $sxl^j\leq q$. This implies $x\leq (q:sl^j)$. Therefore $x\leq (q:sl^n)$. This gives $sxl^n\leq q$. Hence $sx\leq (q:l^n)$. Thus $s(q:l^j)\leq (q:l^n)$. This proves the ascending chain of elements $(q:l) \leq (q:l^2) \leq (q:l^3) \leq \cdots $ is $S$-stationary for all $l \in L$.
\end{proof}	

\begin{prop}
	Let $L_1$ and $L_2$ be two $c$-lattices and $S_1$, $S_2$ be the multiplicatively closed subsets of $L_1$, $L_2$ respectively. Let $S=S_1\times S_2$ be multiplicatively closed subset of  $L=L_1\times L_2$. Then:
	\begin{enumerate}
		\item\label{P3.10.1} $p_1$ is an $S_1$-primary element of $L_1$ if and only if $(p_1,1)$ is an $S$-primary element of $L$.
		\item\label{P3.10.2} $p_2$ is an $S_2$-primary element of $L_1$ if and only if $(1,p_2)$ is an $S$-primary element of $L$.
	\end{enumerate}
\end{prop}    
\begin{proof} 
	(\ref{P3.10.1})	Note that , $t_1\nleq p_1$ for all $t_1\in S_1$ if and only if $(t_1,t_2)\nleq (p_1,1)$ for all $(t_1,t_2) \in S$. Also $t_2 \nleq p_2$ for all $t_2 \in S_2$ if and only if $(t_1,t_2) \nleq (1,p_2)$ for all $(t_1,t_2) \in S$. Let $(a_1,b_1), (a_2,b_2)\in L_1\times L_2$ such that $(a_1,b_1)(a_2,b_2)\leq (p_1,1)$. Then $a_1a_2\leq p_1$. Since $p_1$ is a $S_1$-primary element of $L_1$, there exists $s_1\in S_1$ such that $s_1a_1\leq p_1$ or $s_1a_2\leq \sqrt{p_1}$. Hence $(s_1,1)(a_1,b_1)=(s_1a_1,b_1)\leq (p_1,1)$ or $(s_1,1)(a_2,b_2)=(s_1a_2,b_2)\leq (\sqrt{p_1},1)=\sqrt{(p_1,1)}$. Thus, $(p_1,1)$ is an $S$-primary element of $L$.
	
	Conversely, suppose that $(p_1,1)$ is an $S$-primary element of $L$. Let $a,b\in L_1$ such that $ab\leq p_1$. Therefore $(a,1)(b,1)=(ab,1)\leq (p_1,1)$. Since $(p_1,1)$ is an $S$-primary element of $L$, there exists an element $(s_1,s_2)\in S$ such that $(s_1a,s_2)=(s_1,s_2)(a,1)\leq (p_1,1)$ or $(s_1b,s_2)=(s_1,s_2)(b,1)\leq \sqrt{(p_1,1)}=(\sqrt{p_1},1)$. This implies $s_1a\leq p_1$ or $s_1b\leq \sqrt{p_1}$. Thus $p_1$ is an $S_1$-primary element of $L_1$.
	
	Proof of 	(\ref{P3.10.2}) is same as that of  (\ref{P3.10.1}).
\end{proof}

\section{weakly S-prime elements}
Almahdi et al. \cite{ABT} and Mahdou et al. \cite{MMZ} introduced the notion of weakly $S$-prime ideals in commutative rings with unity.

\begin{defn} 
	Let $R$ be a commutative ring and $S\subseteq R$ be a multiplicative set of $R$. An ideal $P$ of $R$ satisfying $P \cap S =\emptyset$ is said to be weakly $S$-prime if there exists an element $s\in S$ such that, whenever $a, b\in R$, $0\neq ab\in P$ implies $sa\in P$ or $sb\in P$.
\end{defn}

In this section, we introduce and study weakly $S$-prime elements in $V$-lattices.

\begin{definition}
	Let $S$ be a multiplicatively closed subset of a $V$-lattice $L$. A proper element $p$ of $L$ such that $t\nleq p$ for all $t\in S$ is called a weakly S-prime element if there exists an element $s\in S$ such that for all $a,b\in L$ with $0\neq ab\leq p$, implies $sa\leq p$ or $sb\leq  p$.	
\end{definition}

\begin{obs}
	\begin{enumerate}
		\item  If $p$ is a weakly prime element of a $V$-lattice $L$ and $S$ is a multiplicatively closed subset of $L$ such that $t\nleq p$ for all $t\in S$, then $p$ is a weakly $S$-prime element of $L$.
		
		\item Let $L$ be a $V$-lattice and $S=\{1\}$, then the weakly prime elements of $L$ and the weakly $S$-prime elements of $L$ coincide.
		
		\item  If $p$ is an $S$-prime element of a $V$-lattice $L$, then $p$ is a weakly $S$-prime element of $L$. The converse need not  hold in general. 
		
		\item If $p$ is a weakly $S$-prime element of a $V$-lattice $L$, then $p$ need not be a prime element of $L$.
		
	\end{enumerate}
\end{obs}

\begin{exa}
	\begin{enumerate}
		\item  Consider $L=Id(\mathbb{Z}_{24})$ and $S=\{(1), (2),(4),(8)\}$. Then $(6)$ is a weakly $S$-prime element of $L$ but not a weakly prime element of $L$. Moreover, $(6)$ is not a prime element of $L$. 
		
		\item Consider $L=Id(\mathbb{Z}_{10})$ and $S=\{(1)\}$. Then $(0)$ is a weakly $S$-prime element of $L$ but not an $S$-prime element of $L$.
		
		\item Consider $L=Id(\mathbb{Z}_{12})$ and $S=\{(1),(3)\}$. Observe that $(3)$ is a weakly prime element of $L$ but not a weakly $S$-prime element of $L$.
		
	\end{enumerate}
\end{exa}

Following result is due to Almahdi, Bouba and Tamekkante \cite{ABT}.

\begin{thm}[Theorem 2.7, \cite{ABT}] 
	Let $R$ be a commutative ring with identity, $S $ be a multiplicative subset of $R$ and $P$ be an ideal of $R$ disjoint with $S$. Then $P$ is weakly $S$-prime if and only if there exists $s \in S$, such that for all $I, J$ two ideals of $R$, if $0\neq IJ \subseteq P$, then $sI \subseteq P$ or $sJ \subseteq P$.
\end{thm}

\begin{thm}
	Let $R$ be a commutative ring with identity and $S$ a multiplicative subset of $R$. 
	Let $L=Id(R)$ denote the multiplicative lattice of all ideals of $R$, and define 
	$S_L=\{(a)\in L \mid a\in S\}.$
	Then an ideal $P$ of $R$ is a weakly $S$-prime ideal of $R$ if and only if $P$ (viewed as an element of $L$) is a weakly $S_L$-prime element of $L$.
\end{thm}

\begin{proof}
	Suppose $P$ is a weakly $S$-prime ideal of $R$. Then $S\cap P=\varnothing$.  We claim that $t\not\leqq P$ for all $t\in S_L$. Suppose on the contrary that  there exists $t=(a)\in S_L$ with $t\le P$. But then $a\in P\cap S$, a contradiction.  
	Hence $t\not\leqq P$ for all $t\in S_L$.  
	
	Now, let $ I,J\in L$ with $0\neq IJ \leq P$, i.e., $\{0\}\neq IJ\subseteq P$. Since $P$ is weakly $S$-prime, there exists $s\in S$ such that $sI\subseteq P$ or $sJ\subseteq P$.  
	Equivalently, $(s)I\leq  P$ or $(s)J\leq P$ in $L$.  
	Thus $P$ is a weakly $S_L$-prime element of $L$.
	
	Conversely, assume that $P$ is a weakly $S_L$-prime element of $L$.  
	By the definition, $t\not\leq P$ for every $t\in S_L$. If some $a\in S\cap P$, then $(a)\in S_L$ with $(a)\le P$, a contradiction. Hence $S\cap P=\varnothing$.  
	
	Further, since $P$ is weakly $S_L$-prime, there exists $(s)\in S_L$ (for some $s\in S$) such that for all $I,J\in L$,
	$0\neq IJ\leq  P$ implies that $ (s)I\leq P \ \text{ or }\ (s)J\leq P$.
	That is, $sI\subseteq P$ or $sJ\subseteq P$. Thus $P$ is a weakly $S$-prime ideal of $R$.
\end{proof}

\begin{definition}\cite{JS} 
	A reduced multiplicative lattice is defined as a multiplicative lattice in which 0 is the only nilpotent element, that is, if for $a\in L$, $a^n=0$ for some $n\in \mathbb{N}$, then $a=0$.
\end{definition}

\begin{lem}\label{L4.8}
	Let $S$ be a multiplicatively closed subset of a $c$-lattice $L$ and $p$ be a weakly $S$-prime element of $L$.
	\begin{enumerate}
		\item\label{L4.8.1} If $p^2\neq 0$, then $p$ is an $S$-prime element of $L$.
		\item\label{L4.8.2} If $p$ is not an $S$-prime element of $L$, then $\sqrt{p}=\sqrt{0}$. 
		\item\label{L4.8.3} If $L$ is reduced, then $p=0$ or $p$ is an  $S$-prime element of $L$.
		
	\end{enumerate}
\end{lem}
\begin{proof}
	(\ref{L4.8.1}) Let $a,b\in L$ such that $ab\leq p$. If $ab\neq 0$, then we are done. Let $ab=0$. Assume that $ap\neq 0$. Then $a(p\vee b)\neq 0$. Otherwise if $a(p\vee b)=0$, then $0=a(p\vee b)=ap\vee ab=ap\vee 0=ap$, which is contradiction. Therefore $0\neq a(p\vee b)=ap\vee ab\leq p\vee p=p$. Since $p$ is a weakly $S$-prime element, there exists $s\in S$ such that $sa\leq p$ or $s(p\vee b)\leq p$. This implies $sa\leq p$ or $sb\leq s(p\vee b)\leq p$. Hence, $p$ is an $S$-prime element of $L$. Similarly we can prove the result for $bp\neq 0$. Assume that $ap=0=bp$. Then $(p\vee a)(p\vee b)=p^2\vee pa\vee pb\vee ab=p^2\neq 0$. Also $0\neq (p\vee a)(p\vee b)=p^2\leq p$. Since $p$ is a weakly $S$-prime, there exists $s\in S$ such that $s(p\vee a)\leq p$ or $s(p\vee b)\leq p$. Hence $sa\leq s(p\vee a)\leq p$ or $sb\leq s(p\vee b)\leq p$. Thus, $p$ is an $S$-prime element of $L$.
	
	(\ref{L4.8.2}) From (\ref{L4.8.1}), if $p$ is not an $S$-prime element of $L$, then $p^2=0$. Hence, $\sqrt{0}=\sqrt{p^2}=\sqrt{p}$.
	
	(\ref{L4.8.3}) Let $L$ be reduced $c$-lattice. If $p\neq 0$, then $p^2\neq 0$. Hence by (\ref{L4.8.1}), $p$ is an $S$-prime element of $L$.  \end{proof}


\begin{cor}
	Let $S$ be a multiplicatively closed subset of a $c$-lattice $L$ and $p$ be a weakly $S$-prime element of $L$ that is not an $S$-prime element of $L$. If $pq=q$, for $q\in L$, then $q=0$.
\end{cor}
\begin{proof}
	Since $p$ is not  $S$-prime, by (\ref{L4.8.1}) of Lemma \ref{L4.8}, $p^2=0$. Therefore $0=p^2q=p(pq)=pq=q$.
\end{proof}

\begin{thm}\label{T4.10}
	Let $S$ be a multiplicatively closed subset of a $c$-lattice $L$ and $p\in L$ such that $t\nleq p$ for all $t\in S$. Then following are equivalent:
	\begin{enumerate}
		\item\label{T4.10.1} $p$ is a weakly $S$-prime element of $L$.
		\item\label{T4.10.2} There exists $s\in S$ such that $(p:a)=(0:a)$ or $(p:a)\leq (p:s)$ for each $a\nleq (p:s)$.
	\end{enumerate}
\end{thm}
\begin{proof}
	(\ref{T4.10.1}) $\implies$ (\ref{T4.10.2}) Suppose $p$ is weakly $S$-prime element of $L$. There exists $s\in S$ such that for all $c,d\in L$ with $0\neq cd\leq p$, $sc\leq p$ or $sd\leq p$. For this $s$, let $a\nleq (p:s)$. Therefore $sa\nleq p$. Assume that $(p:a)\neq (0:a)$. Then there exists $b\in L_*$ such that $b\leq (p:a)$ and $b\nleq (0:a)$. Therefore $ab\leq p$ and $ab\neq 0$.  To prove $(p:a)\leq (p:s)$, let $x\leq (p:a)$ for $x\in L_*$. Then $xa\leq p$. If $xa\neq 0$, $sa\leq p$ or $sx\leq p$. Since $sa\nleq p$, we have $sx\leq p$. Hence $x\leq (p:s)$. If $ax=0$, then $0\neq ab=ab\vee 0=ab\vee ax=a(b\vee x)$. Since $0\neq ab\leq p$, $0\neq a(b\vee x)\leq p$. Since $sa\nleq p$ and $p$ is weakly $S$-prime, we have $s(b\vee x)\leq p$. Therefore $sx\leq s(b\vee x)\leq p$. Hence $x\leq (p:s)$. Thus, $(p:a)\leq (p:s)$.
	
	(\ref{T4.10.2}) $\implies$ (\ref{T4.10.1}) Suppose there exists $s\in S$ such that $(p:a)=(0:a)$ or $(p:a)\leq (p:s)$ for each $a\nleq (p:s)$. Let $c,d\in L$ with $0\neq cd\leq p$. Assume that $sc\nleq p$. Then $c\nleq (p:s)$. Since $cd\leq p$, we have $d\leq (p:c)$. Also $cd\neq 0$ implies that $d\nleq (0:c)$. So, we have $d\in L$ such that $d\leq (p:c)$ and $d\nleq (0:c)$. Therefore $(p:c)\neq (0:c)$. Since $c\nleq (p:s)$, by (\ref{T4.10.2}), $(p:c)\leq (p:s)$. Hence $d\leq (p:c)\leq (p:s)$ implies $sd\leq p$. Thus, $p$ is a weakly $S$-prime element of $L$.
\end{proof}

\begin{lem}
	Let $S$ be a multiplicatively closed subset of a $c$-lattice $L$. If $p$ is a weakly $S$-prime element of $L$ that is not $S$-prime element of $L$, then $sp\sqrt{0}=0$ for some $s\in S$.
\end{lem}	
\begin{proof}
	Suppose $p$ is a weakly $S$-prime element of $L$ that is not $S$-prime element of $L$.
	By Theorem \ref{T4.10}, there exists $s\in S$ such that for each $a\nleq (p:s)$, either $(p:a)=(0:a)$ or $(p:a)\leq (p:s)$. Let $x$ be compact element of  $L$ such that $x\leq \sqrt{0}$. If $x\leq (p:s)$, then $sx\leq p$. Hence by Lemma \ref{L4.8}(1), $sxp\leq p^2=0$. Therefore $sp\sqrt{0}=0$. Suppose $x\nleq (p:s)$. Then either $(p:x)=(0:x)$ or $(p:x)\leq (p:s)$. Now, if $(p:x)=(0:x)$, then $p\leq (0:x)$. This implies $px=0$. Hence $sxp=0$. Assume that $(p:x)\leq (p:s)$. Since $x\leq \sqrt{0}$, $x^n=0$ for some positive integer $n$. Then $x^{n-1}\leq (p:x)\leq (p:s)$. This implies $sx^{n-1}\leq p$. Clearly $n\geq 2$ as $s\nleq p$. If $sx^{n-1}\neq 0$, then $sx\leq p$ because $p$ is a weakly $S$-prime element of $L$. Hence $x\leq (p:s)$, a contradiction. Therefore $sx^{n-1}=0$. Let $j$ be the smallest integer such that $sx^j=0$. As $sx\neq 0$, we have $j\geq 2$.
	If $sxp\neq 0$, then $0\neq sx(x^{j-1}\vee p)= sxp\leq p$. Since $p$ is weakly $S$-prime element, $s^2x\leq p$ or $s(x^{j-1}\vee p)\leq p$. If $s^2x\leq p$, then $sx\leq p$, a contradiction. Therefore, $s(x^{j-1}\vee p)\leq p$. Hence $0\neq sx^{j-1}\leq p$. Since $p$ is a weakly $S$-prime element, $sx\leq p$, a contradiction. Therefore, $sxp=0$. Thus $sp\sqrt{0}=0$.
\end{proof}

\begin{thm}
	Let $S$ be a multiplicatively closed subset of a $c$-lattice $L$ such that $S\bigcap Z(L)=\emptyset$ and $p \in L$ such that $t \nleq p$ for all $t \in S$. Then following are equivalent:
	
	\begin{enumerate}
		\item $p$ is a weakly $S$-prime element of $L$.
		
		\item There exists $s \in S$ such that for each $x \nleq (p:s)$, $((p:sx)] = ((p:s)] \cup ((0:x)]$ for some $s \in S$.
		
		\item There exists $s \in S$ such that for each $x \nleq (p:s)$, $((p:sx)] = ((p:s)]$ or $((p:sx)] = ((0:x)]$ for some $s \in S$.
	\end{enumerate} 
\end{thm}
\begin{proof}
	$(1) \Rightarrow (2)$:  Suppose $p$ is a weakly $S$-prime element of $L$. Therefore for all $a,b \in L$ with $0 \neq ab \leq p$, there exists $s \in S$ such that $sa \leq p$   or $sb \leq p$. Take this $s$. Let $y$ be a compact element of $L$ such that $y  \in ((p:sx)]$ i.e. $y\leq (p:sx)$, where $x \nleq (p:s)$. If $xy =0$, then $y \leq (0:x)$, i.e. $y \in ((0:x)]$. So suppose that $xy \neq 0$. As $S\bigcap Z(L)=\emptyset$, we have $sxy \neq 0$ and $sxy \leq p$. This gives $s^2 x \leq p$ or  $sy \leq p$. If $s^2 x \leq p$, then we get a contradiction to $x \nleq (p:s)$. Therefore  $sy \leq p$ and this gives $y \leq (p:s)$, i.e. $y \in ((p:s)]$. In either case $y\in ((p:s)] \cup ((0:x) ]$. Therefore $((p:sx)] \subseteq ((p:s)] \cup ((0:x)]$. For the other inclusion, Let $z \in L_*$ such that $z \in ((p:s)] \cup ((0:x)]$. If $z \in ((p:s)] $, i.e. $z \leq (p:s)$, then $zs\leq p$, so $zsx\leq p$ and hence $z \leq (p:sx)$, i.e. $z \in ((p:sx)]$. If $z \in  ((0:x)]$ i.e. $z \leq  (0:x)$, then $zx =0 \leq p$, so $sxz =0 \leq p$ and hence $z \leq (p:sx)$, i.e. $z \in ((p:sx)]$. Therefore $((p:s)] \cup ((0:x)] \subseteq ((p:sx)]$ and hence $((p:sx)] = ((p:s)] \cup ((0:x)]$. 
	
	$(2) \Rightarrow (3)$: It is clear.
	
	$(3) \Rightarrow (1)$: Let $a,b \in L$ such that $0 \neq ab \leq p$. Suppose that $sa \nleq p$ for all $s \in S$. Hence $a \nleq (p:s)$. As $0 \neq ab \leq p$ and  $S\bigcap Z(L)=\emptyset$, we get $0 \neq sab \leq p$. This gives $b \leq (p:sa)$, i.e., $ b \in ((p:sa)] = ((p:s)] \cup ((0:a)]$. Since $ab \neq 0$, we get $b \in ((p:s)]$ i.e. $b \leq (p:s)$ and hence $sb \leq p$. This proves that $p$ is a weakly $S$-prime element of $L$. 
\end{proof}

\begin{lem}
	Let $S$ be a multiplicatively closed subset of a multiplicative lattice $L$ and $p\in L$ such that $t\nleq p$ for all $t\in S$. 
	\begin{enumerate}
		\item\label{L4.13.1} If $(p:s)$ is a weakly prime element of $L$ for some $s\in S$, then $p$ is a weakly $S$-prime element of $L$.
		\item\label{L4.13.2} If $S\bigcap Z(L)=\emptyset$ and $p$ is a non-zero weakly $S$-prime element of $L$, then $(p:s)$ is a weakly prime element of $L$ for some $s\in S$.
	\end{enumerate}
\end{lem}
\begin{proof}
	(\ref{L4.13.1}) Suppose $(p:s)$ is a weakly prime element of $L$ for some $s\in S$. Let $ 0 \neq ab \leq p$. As $p \leq (p:s)$, we have  $0 \neq ab \leq (p:s)$. Therefore $a\leq (p:s)$ or $b\leq (p:s)$. This gives $sa\leq p$ or $sb\leq p$. Thus, $p$ is a weakly $S$-prime element of $L$.
	
	(\ref{L4.13.2}) Let $p$ be a non-zero weakly $S$-prime element of $L$. Then there exists $s\in S$ such that for all $a,b\in L$ with $0 \neq ab \leq p$, $sa\leq p$ or $sb\leq p$. Take this $s$. Let $ c, d \in L$ such that $ 0 \neq cd \leq (p:s)$. Therefore $cds\leq p$. Since $S\bigcap Z(L)=\emptyset$, $cds\neq 0$. Therefore $s^2c\leq p$ or $sd\leq p$. If $s^2c\leq p$. Then $s^2 c = 0$ or $s^2 c \neq 0$. If $s^2c=0$, then as $S\bigcap Z(L)=\emptyset$ we get $c=0$. Hence $cd=0$, a contradiction. Therefore $ 0 \neq s^2 c\leq p$. Hence $s^3\leq p$ or $sc\leq p$. Since $t\nleq p$ for all $t\in S$, we get  $s^3\nleq p$. Therefore $sc\leq p$. Hence $c\leq (p:s)$. If	$sd\leq p$. Then $d\leq (p:s)$.  Thus, $(p:s)$ is a weakly prime element of $L$.
\end{proof}

\begin{defn}[D. D. Anderson \cite{A}]
	Let $a, b, m$ be elements of a multiplicative lattice $L$. 
	\begin{enumerate}
		\item An element $m$ is said to be\textit{ meet principal} if $a \wedge mb = m((a:m) \wedge b)$ for all $a,b \in L$.
		\item An element $m$ is said to be \textit{join principal} if $a \vee (b :m) = (am \vee b): m$ for all $a,b \in L$.
		\item An element $m$ is said to be  \textit{principal} if $m$ is both meet principal and join principal.
	\end{enumerate}
	We say that a multiplicative lattice is \textit{principally generated}, if every nonzero element is a join of principal elements. A multiplicaive lattice $L$ is said to be an \textit{$r$-lattice}, if  $L$ is a modular, principally generated, and compactly generated lattice with $1$ as a compact element of $L$. 
\end{defn}

\begin{rem}
	Let $L$ be $r$-lattice, then every principal element of $L$ is compact, moreover product of two compact elements is compact in $L$. For more details see \cite{A}.
\end{rem}

\begin{defn} [D. D. Aderson and Tiberiu Dumitrescu \cite{AT}]
	Let $S$ be a multiplicative subset of a commutative ring $R$ with 1. An ideal $I$ of $R$ is said to be \textit{$S$-finite}, if there exists some finitely generated ideal $J$ of $R$ and some $s \in S$ such that $sI \subseteq J \subseteq I$. A commutative ring $R$ with 1 is said to be \textit{$S$-Noetherian} if each ideal of $R$ is $S$-finite.  
\end{defn}

Analogously we defined $S$-Noetherian lattice in \cite{SPJ2} as follows.

\begin{defn}\cite{SPJ2}
	Let $S$ be a multiplicative subset of a multiplicative lattice $L$. An element $a$ of $L$ is said to be \textit{$S$-compact}, if there exists some compact element $b$ in $L$ and some $s \in S$ such that $s\cdot a \leq b \leq a$. An $r$-lattice $L$ is said to be an \textit{$S$-Noetherian lattice} if each element of $L$ is $S$-compact.  
\end{defn}

	\begin{thm}
		Let $L$ be an $r$-lattice and $S$ be a multiplicatively closed subset of $L$. Then the following are equivalent:
		\begin{enumerate}
			\item\label{SN1} $L$ is $S$-Noetherian.
			
			\item\label{SN2} Every weakly $S$-prime element of $L$ is $S$-compact.
			
			\item\label{SN3} Every $S$-prime element of $L$ is $S$-compact.
			
		\end{enumerate}
	\end{thm}
	\begin{proof}
		(\ref{SN1}) $\implies$ (\ref{SN2}) see [Theorem 2.11, \cite{SPJ2}]
		
		(\ref{SN2}) $\implies$ (\ref{SN3}) Obvious.
		
		(\ref{SN3}) $\implies$ (\ref{SN1});
		see [Theorem 2.13, \cite{SPJ2}].
	
	\end{proof}


\begin{lem}
	Let $L$ be a $V$-lattice and $S$ and $S'$  be multiplicatively closed subsets of $L$ with  $S \subseteq S'$ and for any $s \in S'$ there exists an element $t \in S'$ such that $st \in S$.  If $p$ is a weakly $S'$-prime element of $L$, then $p$ is a weakly $S$-prime element of $L$.
\end{lem}

\begin{proof} 
	Let $a,b \in L$ such that $0\neq ab \leq p$. Since $p$ is a weakly $S'$- prime element of $L$, there exists an element $s \in S'$ such that $sa \leq p$ or $sb \leq p$. Therefore $st'a \leq p$ or $st'b \leq p$  for all $t' \in S'$. As there exists $t \in S'$ such that $ st \in S$, we get $\bar{s}a \leq p$ or $\bar{s}b \leq p$, where $\bar{s} =s t' \in S$. This proves that $p$ is a weakly $S$-prime element of $L$.  
\end{proof}

\begin{lem}\label{L4.20}
	Let $L$ be a $V$-lattice and $S$ and $S'$  be multiplicatively closed subsets of $L$ such that $S \subseteq S'$.  If $p$ is a weakly $S$-prime element of $L$, then $p$ is a weakly $S'$-prime element of $L$.
\end{lem}

\begin{proof}
	Follows from the definition of a weakly $S$-prime element.
\end{proof}

\begin{rem}
	Converse of the Lemma \ref{L4.20} need not be true in general. Let $L=Id(\mathbb{Z}_{12})$, i.e. $L$ be a ideal lattice of $\mathbb{Z}_{12}$. Then $L$ is a $V$-lattice. If $S =\{1\} \subseteq S' =\{1,2,4\}$ is a multiplicatively closed subsets of $L$, then $(6)$ is a weakly $S'$-prime element of $L$ but $(6)$ is  not a weakly $S$-prime element of $L$. 
\end{rem}

We introduce following definition to prove under which condition he converse of  Lemma \ref{L4.20} will hold.

\begin{definition}
	Let $S$ be a multiplicatively closed subset of a $V$-lattice $L$. The \textit{saturation of $S$} is the set $S^*=\{x\in L_* \;|\; xy\in S \textrm{ for some }y \in L_*\}$. It is easy to verify that $S^*$ is a multiplicatively closed subset of $L$ and $S\subseteq S^* \subseteq L_*$.
\end{definition}

\begin{exa}
	Let $L=Id(\mathbb{Z}_{12})$. If $S =\{(1),(4)\}$ is a multiplicatively closed subset of $L$, then $S^* = \{(1), (2), (4) \}$.  
\end{exa}

\begin{lem}
	Let $S$ be a multiplicatively closed subset of a $V$-lattice $L$ and $p\in L$ such that $t\nleq p$ for all $t\in S$. Then $p$ is a weakly $S$-prime element of $L$ if and only if $p$ is a weakly $S^*$-prime element of $L$, where $S^*$ is a saturation of $S$.
\end{lem}	
\begin{proof}
	Since $S\subseteq S^*$, if $p$ is a weakly $S$-prime element of $L$, then by Lemma \ref{L4.20}, $p$ is a weakly $S^*$-prime element of $L$.
	Conversely, suppose $p$ is a weakly $S^*$-prime element of $L$.  Let $a,b\in L$ such that $0 \neq ab \leq p$. Therefore there exists $s^*\in S^*$ such that $s^*a\leq p$ or $s^*b\leq p$. As $s^*\in S^*$, there exists $y\in L_*$ such that $s^*y\in S$. Let $s=s^*y$. Then $sa=s^*ya\leq s^*a\leq p$ or $sb=s^*yb\leq s^*b\leq p$. Thus, $p$ is a weakly $S$-prime element of $L$.
\end{proof}

\begin{lem}
	Let $S$ be a multiplicatively closed subset of a $V$-lattice $L$. If $p$ is a  weakly $S$-prime element of $L$ and $0$ is an $S$-prime element of $L$, then $p$ is an $S$-prime element of $L$.
\end{lem}
\begin{proof}
	Since $p$ is a  weakly $S$-prime element of $L$, $t\nleq p$ for all $t\in S$. Let $a,b\in L$ with $ab\leq p$.  If $ab\neq 0$, then there exists $s_1\in S$ such that $s_1a\leq p$ or $s_1b\leq p$. Assume that $ab= 0$. Since $0$ is an $S$-prime element of $L$, there exists $s_2\in S$ such that $s_2a= 0$ or $s_2b= 0$. Therefore $s_2a= 0\leq p$ or $s_2b= 0 \leq p$. Choose $s=s_1s_2$. Then for all $a,b \in L$ with $ab\leq p$, implies  $sa\leq p$ or $sb \leq p$. Thus, $p$ is an $S$-prime element of $L$.
\end{proof}

\begin{lem}
	Let $S$ be a multiplicatively closed subset of a $V$-lattice $L$ and $p$ is a  weakly $S$-prime element of $L$. If $q$ is any element of $L$ such that   $t\leq q$ for some $t\in S$, then $pq$ and  $p\wedge q$ are  weakly $S$-prime elements of $L$.
\end{lem}
\begin{proof} 
	Since $t \nleq p$ for all $t \in S$, we have  $t\nleq p\wedge q$  and $t \nleq pq$ for all $t\in S$. let $c,d\in L$ with $0 \neq cd\leq p q$. Then $ 0 \neq cd\leq p$. Since $p$ is a  weakly $S$-prime element of $L$, there exists $s\in S$ such that $sc\leq p$ or $sd\leq p$. Take $t'\in S$ such that $t'\leq q$. Then $t'sc\leq pq $ or $t'sd\leq pq $. Hence, $p q$ is a weakly $S$-prime element of $L$.
	
	Now, let $a,b\in L$ with $0 \neq ab\leq p\wedge q$. Then $ 0 \neq ab\leq p$. Since $p$ is a  weakly $S$-prime element of $L$, there exists $s'\in S$ such that $s'a\leq p$ or $s'b\leq p$. Take $t^{''}\in S$ such that $t^{''}\leq q$. Then $t^{''}s'a\leq pq \leq  p\wedge q$ or $t^{''}s'b\leq pq \leq  p\wedge q$. Hence, $p\wedge q$ is a weakly $S$-prime element of $L$. 
\end{proof}

\begin{lem}
	Let $S$ be a multiplicatively closed subset of a $V$-lattice $L$ and $i,~j$ be two elements of $L$ such that $i \in L_*$ and $i \leq j$.
	
	\begin{enumerate}
		\item If $j$ is a weakly $S$-prime element of $L$, then $j$ is a weakly $\bar{S}$-prime element of $L/i$,  where $\bar{S} = \{s \vee i \;|\; s\in S\}$.
		
		\item If $j$ is a weakly $\bar{S}$-prime element of $L/i$ and $i$ is an $S$-prime element of $L$, then $j$ is an $S$-prime element of $L$.
		
		\item If $j$ is a weakly $\bar{S}$-prime element of $L/i$ and $i$ is a weakly  $S$-prime element of $L$, then $j$ is a weakly   $S$-prime element of $L$.
	\end{enumerate}
\end{lem}

\begin{proof}
	\begin{enumerate}
		\item Firstly, we will show that $\bar{S}$ is a multiplicatively closed subset of $L/i$. Since $i$ is compact and every $s \in S$ is compact, $s \vee i$ is compact for every $s \in S$. Therefore $\bar{S} \subseteq  (L/i)_*$. Let $s_1 \vee i, ~ s_2 \vee i \in \bar{S} $. therefore $(s_1 \vee i) \circ (s_2 \vee i) = s_1 s_2 \vee s_1 i \vee s_2 i \vee i^2 \vee i = s_1 s_2 \vee  i $. This gives $(s_1 \vee i) \circ (s_2 \vee i) \in \bar{S}$. Hence $\bar{S}$ is multiplicatively closed subset of $L/i$. We claim that $t' \nleq j$ for all $t' \in \bar{S}$. Suppose on the contrary that $t'_1 \leq j$ for some $t'_1 \in \bar{S}$. Therefore $t_1 \vee i \leq j$, where $t'_1 = t_1 \vee i$ and $t_1 \in S$. This gives $t_1 \leq j$, a contradiction to  $t \nleq j$ for all $t \in S$, as $j$ weakly is a $S$-prime element of $L$. Hence  $t' \nleq q$ for all $t' \in \bar{S}$.  Let $x, y \in L/i$ with $i\neq x \circ y \leq j$. 
		Therefore  $x \circ y  = xy \vee i \leq j$. This gives $xy \leq j$. If $xy=0$, then $x \circ y=xy\vee i=0\vee i=i$, a contradiction. Therefore $xy\neq 0$. As $j$ is a weakly $S$-prime element of $L$, there exists $s \in S$ such that $sx \leq j$ or $sy \leq j$.  If $sx \leq j$, we get $(s \vee i) x = sx \vee ix \leq j$, as $i \leq j$.  If $sy \leq j$, we get $(s \vee i) y = sy \vee iy \leq j$, as $i \leq j$.  
		Let $xy=0$.
		Thus, $j$ is an $\bar{S}$-primary element of $L/i$.
		
		\item Suppose that  $j$ is a weakly $\bar{S}$-prime element of $L/i$ and $i$ is an $S$-prime element of $L$. Since $j$ is a weakly $\bar{S}$-prime element of $L/i$, $t \nleq j$ for all $t \in \bar{S}$. If $s\leq j$ for some $s\in S$, then $t=s\vee i\leq j$, where $t=s\vee i\in \bar{S}$, a contradiction. Hence, $s\nleq j$ for all $s\in S$. Let $a,~b \in L$ such that $ab \leq j$. Since $i \leq j$, we get $(a \vee i ) \circ(b \vee i) = ab \vee ai \vee bi \vee i^2\vee i  = ab \vee i \leq j$. There are two cases. 
		
		First case is $ab \nleq i$ and second case is $ab \leq i$. Consider the first case $ab \nleq i$. This gives $ab \vee i \neq i$ in $L/i$.  As $(a \vee i), (b \vee i) \in L/i$ such that $(a \vee i ) \circ(b \vee i)\leq j$ and $j$ is a weakly $\bar{S}$-prime element of $L/i$.  There exists $s \vee i \in \bar{S}$ such that $(s \vee i) \circ(a \vee i) \leq j$ or $(s \vee i)\circ (b \vee i) \leq j$. This gives $sa \vee si \vee ai \vee i^2\vee i \leq j$ or $sb \vee si \vee bi \vee i^2 \vee i\leq j$.  Therefore $sa \leq sa \vee i \leq j $ or $sb \leq sb \vee i \leq j$. Now, Consider the  second case $ab \leq i$.  Since $i$ is an $S$-prime element of $L$, there exists $s' \in S$ such that $s'a \leq i \leq j$ or $ s'b \leq i \leq j$. Hence in either case $j$ is an $S$-prime element of $L$.
		
		\item Proof is similar to $(2)$.  
	\end{enumerate}
\end{proof}


\begin{rem}
	If $p_1$ is a weakly $S_1$-prime element of $V$-lattice $L_1$ and $p_2$ is a weakly $S_2$-prime element of $V$-lattice $L_2$, then $(p_1,p_2)$ need not be a weakly $(S_1\times S_2)$-prime element of $L_1\times L_2$. For this, consider $L_1=L_2=Id(\mathbb{Z})$ and $S_1=S_2=\{\mathbb{Z}\}$. Then $p_1=9\mathbb{Z}$ is a weakly $S_1$-prime element of $L_1$ and $p_2=4\mathbb{Z}$ is a weakly $S_2$-prime element of $L_2$. However, $(p_1,p_2)=(9\mathbb{Z},4\mathbb{Z})$ is not a weakly $(S_1\times S_2)$-prime element of $L_1\times L_2$ as $(0,0)\neq (9\mathbb{Z},\mathbb{Z})(\mathbb{Z},4\mathbb{Z})=(9\mathbb{Z},4\mathbb{Z})\leq (9\mathbb{Z},4\mathbb{Z})=(p_1,p_2)$, but $(9\mathbb{Z},\mathbb{Z})\nleq (p_1,p_2)$ and $(\mathbb{Z},4\mathbb{Z})\nleq (p_1,p_2)$. Thus, $(p_1,p_2)$ is not a weakly $(S_1\times S_2)$-prime element of $L_1\times L_2$.
\end{rem}

\begin{thm}
	Let $S_1$, $S_2$ be the multiplicatively closed subsets of c-lattices $L_1$, $L_2$ repectively and $p_1\in L_1$, $p_2\in L_2$ be the non-zero elements. Let $L=L_1\times L_2$, $S=S_1\times S_2$ and $p=(p_1,p_2)$. Then the following are equivalent.
	\begin{enumerate}
		\item\label{T4.29.1} $p$ is a weakly $S$-prime element of $L$.
		\item\label{T4.29.2} $p_1$ is an $S_1$-prime element of $L_1$ and $t_2\leq p_2$ for some $t_2\in S_2$ or $p_2$ is an $S_2$-prime element of $L_2$ and $t_1\leq p_1$ for some $t_1\in S_1$.
		\item\label{T4.29.3} $p$ is an $S$-prime element of $L$.
	\end{enumerate}
\end{thm}
\begin{proof}
	(\ref{T4.29.1}) $\implies$ (\ref{T4.29.2}) Assume on the contrary that $t_1\nleq p_1$ for all $t_1\in S_1$ and $t_2\nleq p_2$ for all $t_2\in S_2$. Since $p_1$ and $p_2$ are non-zero, choose a non-zero element $(a,b)\leq p$. Then $(0,0)\neq (a,1)(1,b)=(a,b)\leq p$. Since $p$ is a weakly $S$-prime element of $L$, there exists $s=(s_1,s_2)\in S$ such that $s(a,1)\leq p$ or $s(1,b)\leq p$. If $s(a,1)\leq p$, then $s_1a\leq p_1$ and $s_2\leq p_2$, a contradiction to our assumption. If $s(1,b)\leq p=(p_1,p_2)$, then $s_1\leq p_1$ and $s_2b\leq p_2$, a contradiction to our assumption. Hence, $t_1\leq p_1$ for some $t_1\in S_1$ or $t_2\leq p_2$ for some $t_2\in S_2$. Without loss of generality, assume that $t_1\leq p_1$ for some $t_1\in S_1$. Since $p$ is a weakly $S$-prime element, $t\nleq p$ for all $t\in S$. Therefore $(t_1,t_2)\nleq (p_1,p_2)$ for all $t_2\in S_2$. As $t_1\leq p_1$, we have $t_2\nleq p_2$ for all $t_2\in S_2$. Suppose $c,d\in L_2$ with $cd\leq p_2$. Choose $0\neq u\in S_1$ such that $u\leq p_1$. Then $(0,0)\neq (u,c)(1,d)\leq (p_1,p_2)=p$. This implies $s(u,c)=(s_1u,s_2c)\leq p$ or $s(1,d)=(s_1,s_2d)\leq p$. Hence $s_2c\leq p_2$ or $s_2d\leq p_2$. Thus, $p_2$ is an $S_2$-prime element of $L_2$.
	
	(\ref{T4.29.2}) $\implies$ (\ref{T4.29.3}) Assume that $p_1$ is an $S_1$-prime element of $L_1$ and $t_2\leq p_2$ for some $t_2\in S_2$. Let $a_1,b_1\in L_1$ and $a_2,b_2\in L_2$ be such that $(a_1,a_2)(b_1,b_2)\leq p$. Then $a_1b_1\leq p_1$. Since $p_1$ is an  $S_1$-prime element, there exists $t_1\in S_1$ such that $t_1a_1\leq p_1$ or $t_1b_1\leq p_1$. Let $t=(t_1,t_2)$. Then $t(a_1,a_2)=(t_1a_1,t_2a_2)\leq (p_1,p_2)=p$ or $t(b_1,b_2)=(t_1b_1,t_2b_2)\leq (p_1,p_2)=p$. Thus, $p$ is an $S$-prime element of $L$.
	
	(\ref{T4.29.3}) $\implies$ (\ref{T4.29.1}) Obvious.
\end{proof}

\begin{cor}
	Let $S_1, S_2, \dots , S_n$ be the multiplicatively closed subsets of $V$-lattices $L_1, L_2, \dots , L_n$ respectively and $p=(p_1,p_2,\dots,p_n)\in L$. Let $L=L_1\times L_2\times \dots \times L_n$ and $S=S_1\times S_2\times \dots \times S_n$. Then the following are equivalent.
	\begin{enumerate}
		\item
		$p$ is a weakly $S$-prime element of $L$.
		\item
		For some $j\in \{1,2,\dots,n\}$, $p_j$ is an $S_j$-prime element of $L_j$ and for all $i\in \{1,2,\dots,n\}\setminus \{j\}$, there exists $t_i\in S_i$ such that $t_i\leq p_i$.
	\end{enumerate}
\end{cor}

	\section{weakly S-primary elements}
	The notion of weakly $S$‑primary ideals in commutative rings with unity introduced by Celikel and Khashan \cite{CK} and Massaoud and Gouaid \cite{MG}.
	
	\begin{defn} 
		Let $R$ be a commutative ring and $S\subseteq R$ be a multiplicative set of $R$. An ideal $P$ of $R$ satisfying $P \cap S =\emptyset$ is said to be weakly $S$-primary if there exists an element $s\in S$ such that, whenever $a, b\in R$, $0\neq ab\in P$ implies $sa\in P$ or $sb\in \sqrt{P}$.
	\end{defn}	
	
	In this section, we introduce and study weakly $S$-primary elements in $c$-lattices.
	
	\begin{definition}
		Let $S$ be a multiplicatively closed subset of a $c$-lattice $L$. A proper element $q$ of $L$ such that $t\nleq q$ for all $t\in S$ is called a  \textit{weakly $S$-primary element} if there exists an element $s\in S$ such that for all $c,d \in L$ with $0 \neq cd\leq q$, implies $sc \leq q$ or $sd \leq \sqrt q$.	
	\end{definition}
	
	\begin{obs}
		\begin{enumerate}
			\item  If $q$ is a weakly primary element of a $c$-lattice $L$ and $S$ is a multiplicatively closed subset of $L$ such that $t\nleq q$ for all $t\in S$, then $q$ is a weakly $S$-primary element of $L$.
			
			\item Let $L$ be a $c$-lattice and $S=\{1\}$, then the weakly primary elements of $L$ and the weakly $S$-primary elements of $L$ coincide.
			
			\item  If $q$ is an $S$-primary element of a $c$-lattice $L$, then $q$ is a weakly $S$-primary element of $L$.
			
			\item If $q$ is a weakly $S$-prime element of a $c$-lattice $L$, then $q$ is a weakly $S$-primary. The converse need not  hold in general. 
			
		\end{enumerate}
	\end{obs}
	
	\begin{exa}
		\begin{enumerate}
			\item  Consider $L=Id(\mathbb{Z}_{24})$ and $S=\{(1), (2),(4),(8)\}$. Observe that $(6)$ is a weakly $S$-primary element of $L$ but not weakly primary element of $L$.
			
			\item Consider $L=Id(\mathbb{Z}_{10})$ and $S=\{(1)\}$. Observe that $(0)$ is a weakly $S$-primary element of $L$ but not an $S$-primary element of $L$.
			
			\item Consider $L=Id(\mathbb{Z}_{12})$ and $S=\{(1),(3)\}$. Observe that $(4)$ is a weakly $S$-primary element of $L$ but not a weakly $S$-prime element of $L$.
		\end{enumerate}
	\end{exa}

\begin{thm}[\cite{CK}, Theorem 2]\label{T5.5} Let $S$ be a multiplicatively closed subset of a ring $R$ and $I$ be an ideal of $R$ disjoint with $S$. Then following assertions are equivalent.
	
	\begin{enumerate}
		\item $I$ is a weakly $S$-primary ideal of $R$.
		
		\item There exists $s \in S$ such that for any two ideals $J, K$ of $R$, if $0 \neq JK \subseteq I $, then $sJ \subseteq \sqrt{I}$ or $s K \subseteq I$. 
	\end{enumerate}  
	
\end{thm}

\begin{thm}
	Let $R$ be a commutative ring with identity and $S$ be a multiplicatively closed subset of $R$ disjoint from $S$. If $Q$ is a weakly $S$-primary ideal of $R$. Then $Q$ is a weakly $S_L$-primary element of $L = Id(R)$, where $S_L = \{(s) \in Id(R) \;|\; s \in S\}$, a multiplicatively closed subset of $L$.
\end{thm}
\begin{proof}
	Since $Q \cap S=\emptyset$, we have $(s)\nsubseteq Q$ for all $s\in S$. Let $I,J\in L$ such that $0\neq IJ\subseteq Q$. Since $Q$  is a weakly $S$-primary ideal of $R$, then by Theorem \ref{T5.5},  there exists $s\in S$ such that $sI\subseteq Q$ or $sJ\subseteq \sqrt{Q}$. Therefore $(s)I\subseteq Q$ or $(s)J\subseteq \sqrt{Q}$. Thus, $Q$ is a weakly $S_L$-primary element of $L$. 
\end{proof}

\begin{prop}\label{P5.7}
	Let $S$ be a multiplicatively closed subset of a $c$-lattice $L$ and $q$ be a weakly $S$-primary element of $L$.
	\begin{enumerate}
		\item\label{P5.7.1} If $q^2\neq 0$, then $q$ is an $S$-primary element of $L$.
		\item\label{P5.7.2} If $q$ is not an $S$-primary element, then $\sqrt{q}=\sqrt{0}$. 
		\item\label{P5.7.3} If $L$ is reduced, then $q=0$ or $q$ is an  $S$-primary element of $L$.
		
	\end{enumerate}
\end{prop}
\begin{proof}
	(\ref{P5.7.1}) Let $c,d \in L$ such that $cd \leq q$. If $cd \neq 0$, then we are done. Let $cd=0$. Assume that $cq \neq 0$. Then $c(q \vee d) \neq 0$. Otherwise if $c (q \vee d)=0$, then $0=c(q\vee d)=cq \vee cd = cq\vee 0=cq$, which is contradiction. Therefore $0 \neq c(q\vee d)=cq\vee cd\leq q\vee q=q$. Since $q$ is a weakly $S$-primary element, there exists $s\in S$ such that $sc\leq \sqrt{q}$ or $s(q\vee d)\leq q$. This implies $sc\leq \sqrt{q}$ or $sd\leq s(q\vee d)\leq q$. Hence, $q$ is an $S$-primary element of $L$. Similarly we can prove the result for $dq\neq 0$. Assume that $cq=0=dq$. Then $(q\vee c)(q\vee d)=q^2\vee  qc\vee  qd\vee  cd =q^2\neq 0$. Also $0\neq (q\vee c)(q \vee d)=q^2\leq q$. Since $q$ is a weakly $S$-primary element, there exists $s\in S$ such that $s(q\vee c)\leq \sqrt{q}$ or $s(q\vee d)\leq q$. Hence $sc\leq s(q\vee c)\leq \sqrt{q}$ or $sd \leq s(q \vee d)\leq q$. Thus, $q$ is an $S$-primary element of $L$.
	
	(\ref{P5.7.2}) From (\ref{P5.7.1}), if $q$ is not an $S$-primary element, then $q^2=0$. Hence, $\sqrt{0}=\sqrt{q^2}=\sqrt{q}$.
	
	(\ref{P5.7.3}) Let $L$ be reduced $c$-lattice. If $q \neq 0$, then $q^2\neq 0$. Hence by (\ref{P5.7.1}), $q$ is an $S$-primary element of $L$. If $q$ is not an $S$-primary element of $L$, then by (\ref{P5.7.1}), $q^2 = 0$. Since $L$ is reduced lattice, we get $q=0$. \end{proof}

\begin{cor}
	Let $S$ be a multiplicatively closed subset of a $c$-lattice $L$ and $q$ be a weakly $S$-primary element of $L$ that is not an $S$-primary element of $L$. If $qp=p$, for $p\in L$, then $p=0$.
\end{cor}
\begin{proof}
	Since $q$ is not an $S$-primary element of $L$, by (\ref{P5.7.1}) of Proposition \ref{P5.7}, $q^2=0$. Therefore $0=q^2p=q(qp)=qp=p$.
\end{proof}

\begin{thm}
	Let $S$ be a multiplicatively closed subset of a $c$-lattice $L$ and $q \in L$ such that $t\nleq q$ for all $t\in S$. Then following are equivalent:
	\begin{enumerate}
		\item\label{T5.9.1} $q$ is a weakly $S$-primary element of $L$.
		\item\label{T5.9.2} There exists $s\in S$ such that $(q:a)=(0:a)$ or $(q:a)\leq (q:s)$ for each $a \nleq (\sqrt{q}:s)$.
	\end{enumerate}
\end{thm}	
\begin{proof}
	(\ref{T5.9.1}) $\implies$ (\ref{T5.9.2}) Suppose $q$ is a weakly $S$-primary element of $L$. There exists $s\in S$ such that for all $c,d\in L$ with $0 \neq cd\leq q$, $sc\leq q$ or $sd\leq \sqrt{q}$. For this $s$, let $a\nleq (\sqrt{q}:s)$. Therefore $sa\nleq \sqrt{q}$. Assume that $(q:a)\neq (0:a)$. Then there exists $b\in L_*$ such that $b\leq (q:a)$ and $b\nleq (0:a)$. Therefore $ab\leq q$ and $ab\neq 0$. Since $sa\nleq \sqrt{q}$ and $q$ is a weakly $S$-primary element of $L$, $sb\leq q$. Now, we prove that  $(q:a)\leq (q:s)$.  Let $x \in L_*$ such that  $x  \leq (q:a)$. Then $xa\leq q$. If $xa\neq 0$, $sa\leq \sqrt{q}$ or $sx\leq q$. Since $sa\nleq \sqrt{q}$, we get $sx\leq q$. Hence $x\leq (q:s)$. If $ax=0$, then $0\neq ab=ab\vee 0=ab\vee ax=a(b\vee x)$. Since $0\neq ab\leq q$, $0\neq a(b\vee x)\leq q$. Since $sa\nleq \sqrt{q}$ and $q$ is a weakly $S$-primary element, we get $s(b\vee x)\leq q$. Therefore $sx\leq s(b\vee x)\leq q$. Hence $x\leq (q:s)$.  So in either case every compact element $\leq (q:a)$ is $\leq  (q:s)$. As $L$ is a $c$-lattice, we have $(q:a)\leq (q:s)$ for each $a\nleq (\sqrt{q}:s)$.
	
	Suppose that  $(q:a)\nleq (q:s)$ for each $a\nleq (\sqrt{q}:s)$. Therefore there exists $w \in L_*$ such that $w \leq (q:a)$ and $w \nleq (q:s)$, i.e. $aw \leq q$ and $ws \nleq q$. As $a\nleq (\sqrt{q}:s)$, we have $sa \nleq \sqrt{q}$. We want to show that $(q:a) =(0:a)$. Let $y \in L_*$ such that $y \leq (0:a)$, this gives $ay =0$. Since $q$ is a weakly $S$-primary element of $L$, $q$ is a proper element with $0\leq q$. Hence $ay =0 \leq q$ and this gives $y \leq (q:a)$. Therefore every compact element $\leq (0:a)$ is $\leq (q:a)$. As $L$ is a $c$-lattice, $(0:a) \leq (q:a)$. Now Let $z 	\in L_*$ such that $z \leq (q:a)$. Therefore $az \leq q$. If $az =0$, then $z \leq (0:a)$. If $az \neq 0$, then $az \leq q$  and $q$ is a weakly $S$-primary element such that $as \nleq \sqrt{q}$ gives $zs \leq q$, i.e. $z \leq (q:s)$. This show that  every compact element $\leq (q:a)$ is $\leq (q:s)$, a contradiction to our assumption that $(q:a)\nleq (q:s)$ for each $a\nleq (\sqrt{q}:s)$. Hence $az =0$ and hence every compact element $\leq (q:a)$ is $\leq (0:a)$. Therefore $(q:a) \leq (0:a)$. Hence $(q:a) = (0:a)$.
	
	(\ref{T5.9.2}) $\implies$ (\ref{T5.9.1}) suppose there exists $s\in S$ such that $(q:a)=(0:a)$ or $(q:a)\leq (q:s)$ for each $a\nleq (\sqrt{q}:s)$. Let $c,d\in L$ with $0 \neq cd \leq q$. Assume that $sc\nleq \sqrt{q}$. Therefore $c\nleq (\sqrt{q}:s)$. Since $cd\leq q$, $d\leq (q:c)$. Also $cd\neq 0$ implies that $d\nleq (0:c)$. So, we have $d\in L$ such that $d\leq (q:c)$ and $d\nleq (0:c)$. Therefore $(q:c)\neq (0:c)$. Since $sc\nleq \sqrt{q}$, by (\ref{T5.9.2}), $(q:c)\leq (q:s)$. Hence $d\leq (q:c)\leq (q:s)$ implies $sd\leq q$. Thus, $q$ is a weakly $S$-primary element of $L$.
\end{proof}

\begin{thm}
	Let $S$ be a multiplicatively closed subset of a $c$-lattice $L$ such that $S\bigcap Z(L)=\emptyset$ and $q \in L$ such that $t \nleq q$ for all $t \in S$. Then following are equivalent:
	
	\begin{enumerate}
		\item $q$ is a weakly $S$-primary element of $L$.
		
		\item There exists $s \in S$ such that for each $x \nleq \sqrt{(q:s)}$, $((q:sx)] = ((q:s)] \cup ((0:x)]$ for some $s \in S$.
		
		\item There exists $s \in S$ such that for each $x \nleq \sqrt{(q:s)}$, $((q:sx)] = ((q:s)]$ or $((q:sx)] = ((0:x)]$ for some $s \in S$.
	\end{enumerate} 
\end{thm}
\begin{proof}
	$(1) \Rightarrow (2)$:  Suppose $q$ is a weakly $S$-primary element of $L$. Therefore for all $c,d \in L$ with $0 \neq cd \leq q$, there exists $s \in S$ such that $sc \leq q$   or $sd \leq \sqrt{q}$. Take this $s$. Let $(y \in L_*) \in ((q:sx)]$ i.e. $y\leq (q:sx)$, where $x \nleq \sqrt{(q:s)}$. If $xy =0$, then $y \leq (0:x)$, i.e. $y \in ((0:x)]$. So suppose that $xy \neq 0$. As $S\bigcap Z(L)=\emptyset$, $sxy \neq 0$ and $sxy \leq q$. This gives $s^2 x \leq \sqrt{q}$ or  $sy \leq q$. If $s^2 x \leq \sqrt{q}$, then $(s^{2n} x^n \neq 0) \leq q$ for some positive integer $n$. Therefore either $s^{2n+1} \leq \sqrt{q}$ or $sx^n \leq q$. Since  $t \nleq q$ for all $t \in S$,  $s^{2n+1} \leq \sqrt{q}$ is not possible. As  $x \nleq \sqrt{(q:s)}$, $sx^n \leq q$ does not hold. Hence $sy \leq q$ and this gives $y \leq (q:s)$, i.e. $y \in ((q:s)]$. In either case $y\in ((q:s)] \cup ((0:x) ]$. Therefore $((q:sx)] \subseteq ((q:s)] \cup ((0:x)]$. For the other inclusion, Let $z \in L_*$ such that $z \in ((q:s)] \cup ((0:x)]$. If $z \in ((q:s)] $, i.e. $z \leq (q:s)$, then $zs\leq q$, so $zsx\leq q$ and hence $z \leq (q:sx)$, i.e. $z \in ((q:sx)]$. If $z \in  ((0:x)]$ i.e. $z \leq  (0:x)$, then $sx =0 \leq q$, so $sxz =0 \leq q$ and hence $z \leq (q:sx)$, i.e. $z \in ((q:sx)]$. Therefore $((q:s)] \cup ((0:x)] \subseteq ((q:sx)]$ and hence $((q:sx)] = ((q:s)] \cup ((0:x)]$. 
	
	$(2) \Rightarrow (3)$: It is clear.
	
	$(3) \Rightarrow (1)$: Let $a,b \in L$ such that $0 \neq ab \leq q$. Suppose that $sa \nleq \sqrt{q}$ for all $s \in S$. Therefore $s^n a^n \nleq q$ for every positive integer $n$. So $s a^n \nleq q$, this gives $a^n \nleq (q:s)$ and hence $a \nleq \sqrt{(q:s)}$. As $(ab \neq 0) \leq q$ and  $S\bigcap Z(L)=\emptyset$, we get $(sab \neq 0) \leq q$. This gives $b \leq (q:sa)$ i.e.$ b \in ((q:sx)] = ((q:s)] \cup ((0:x)]$. Since $ab \neq 0$, we get $b \in ((q:s)]$ i.e. $b \leq (q:s)$ and hence $sb \leq q$. This proves that $q$ is a weakly primary element of $L$. 
\end{proof}

\begin{prop}
	Let $S$ be a multiplicatively closed subset of a $c$-lattice $L$ and $q \in L$ such that $t\nleq q$ for all $t\in S$. 
	\begin{enumerate}
		\item\label{P5.11.1} If $(q:s)$ is a weakly primary element of $L$ for some $s\in S$, then $q$ is a weakly $S$-primary element of $L$.
		\item\label{P5.11.2} If $S\bigcap Z(L)=\emptyset$ and $q$ is a non-zero weakly $S$-primary element of $L$, then $(q:s)$ is a weakly primary element of $L$ for some $s\in S$.
	\end{enumerate}
\end{prop}
\begin{proof}
	(\ref{P5.11.1}) Suppose $(q:s)$ is a weakly primary element of $L$ for some $s\in S$. Let $0 \neq ab \leq q$. As $q \leq (q:s)$, we have  $(ab \neq 0 ) \leq (q:s)$. Therefore $a\leq (q:s)$ or $b\leq \sqrt{(q:s)}$. This implies $sa\leq q$ or $b^ms\leq q$ for some positive integer $m$. Hence $sa\leq q$ or $(bs)^m\leq b^ms\leq q$ for some positive integer $m$. This proves $sa\leq q$ or $sb\leq \sqrt{q}$. Thus, $q$ is a weakly $S$-primary element of $L$.
	
	(\ref{P5.11.2}) Let $q$ be a non-zero weakly $S$-primary element of $L$. Then there exists $s\in S$ such that for all $a,b\in L$ with $0 \neq ab \leq q$, $sa\leq \sqrt{q}$ or $sb\leq q$. Take this $S$. Let $ c, d \in L$ such that $ cd \neq 0 $ and   $ cd \leq (q:s)$. Therefore $cds\leq q$. Since $S\bigcap Z(L)=\emptyset$, $cds\neq 0$. Therefore $s^2c\leq q$ or $sd\leq \sqrt{q}$.
	
	Case I: If $s^2c\leq q$. Then $s^2 c = 0$ or $s^2 c \neq 0$. If $s^2c=0$, then as $S\bigcap Z(L)=\emptyset$ we get $c=0$. Hence $cd=0$, a contradiction. Therefore $ 0 \neq s^2 c\leq q$. Hence $s^3\leq \sqrt{q}$ or $sc\leq q$. Since $t\nleq q$ for all $t\in S$, we get  $s^3\nleq \sqrt{q}$. Therefore $sc\leq q$. Hence $c\leq (q:s)$.
	
	Case II: If	$sd\leq \sqrt{q}$. Then $s^nd^n=(sd)^n\leq q$ for some positive integer $n$. If $s^nd^n=0$, then as $S\bigcap Z(L)=\emptyset$, we get $d^n=0$. Hence $d\leq \sqrt{0}\leq \sqrt{(q:s)}$. So $s^nd^n\neq 0$. Since $q$ is a  weakly $S$-primary element of $L$, $s^{n+1}\leq \sqrt{q}$ or $sd^n\leq q$. As $s^{n+1}\nleq \sqrt{q}$, we get $sd^n\leq q$. This implies $d^n\leq (q:s)$. Hence $d\leq \sqrt{(q:s)}$. Thus, $(q:s)$ is a weakly primary element of $L$.
\end{proof}

\begin{prop}
	Let $L$ be a $c$-lattice and $S$ and $S'$  be multiplicatively closed subsets of $L$ with  $S \subseteq S'$ and for any $s \in S'$ there exists an element $t \in S'$ such that $st \in S$.  If $q$ is a weakly $S'$-primary element of $L$, then $q$ is a weakly $S$-primary element of $L$.
\end{prop}

\begin{proof} Let $c,d \in L$ such that $0 \neq cd \leq q$. Since $q$ is a weakly $S'$- primary element of $L$, there exists an element $s \in S'$ such that $sc \leq q$ or $sd \leq \sqrt{q}$. Therefore $st'c \leq q$ or $st'd \leq \sqrt{q}$  for all $t' \in S'$. As there exists $t \in S'$ such that $ st \in S$, we get $\bar{s}c \leq q$ or $\bar{s}d \leq \sqrt{q}$, where $\bar{s} =s t' \in S$. This proves that $q$ is a weakly $S$-primary element of $L$.  
\end{proof}

\begin{prop}\label{imply}
	Let $L$ be a $c$-lattice and $S$ and $S'$  be multiplicatively closed subsets of $L$ such that $S \subseteq S'$.  If $q$ is a weakly $S$-primary element of $L$, then $q$ is a weakly $S'$-primary element of $L$.
\end{prop}

\begin{proof}
	Follows from the definition of a weakly $S$-primary element.
\end{proof}

\begin{rem}
	Converse of the Proposition \ref{imply} need not be true in general. Let $L=Id(\mathbb{Z}_{12})$, an ideal lattice of ring $\mathbb{Z}_{12}$. Then $L$ is a $c$-lattice. If $S =\{1\} \subseteq S' =\{1,2,4\}$ is a multiplicatively closed subsets of $L$, then $(6)$ is a weakly $S'$-primary element of $L$, but $(6)$ is  not a weakly $S$-primary element of $L$. 
\end{rem}

\begin{prop}
	Let $S$ be a multiplicatively closed subset of a $c$-lattice $L$ and $q \in L$ such that $t\nleq q$ for all $t\in S$. Then $q$ is a weakly $S$-primary element of $L$ if and only if $q$ is a weakly $S^*$-primary element of $L$, where $S^*$ is a saturation of $S$.
\end{prop}	
\begin{proof}
	Since $S\subseteq S^*$, if $q$ is a weakly $S$-primary element of $L$, then by Proposition \ref{imply}, $q$ is a weakly $S^*$-primary element of $L$.
	Conversely, suppose $q$ is a weakly $S^*$-primary element of $L$. Therefore there exists $s^*\in S^*$ such that for all $c,d\in L$ such that $ 0 \neq cd\leq q$, $s^*c\leq q$ or $s^*d\leq \sqrt{q}$. Let $a,b\in L$ such that $0 \neq ab \leq q$. Therefore $s^*a\leq q$ or $s^*b\leq \sqrt{q}$. As $s^*\in S^*$, there exists $y\in L_*$ such that $s^*y\in S$. Let $s=s^*y$. Then $sa=s^*ya\leq s^*a\leq q$ or $sb=s^*yb\leq s^*b\leq \sqrt{q}$. Thus, $q$ is a weakly $S$-primary element of $L$.
\end{proof}

\begin{prop}
	Let $S$ be a multiplicatively closed subset of a $c$-lattice $L$. If $q$ is a  weakly $S$-primary element of $L$ and $0$ is an $S$-primary element of $L$, then $\sqrt{q}$ is an $S$-prime element of $L$.
\end{prop}
\begin{proof}
	Since $q$ is a  weakly $S$-primary element of $L$, $t\nleq q$ for all $t\in S$.
	If $t\leq \sqrt{q}$ for some $t\in S$, then $t^n\leq q$ for some positive integer $n$, a contradiction to $t\nleq q$ for all $t\in S$. Therefore $t\nleq \sqrt{q}$ for all $t\in S$. Let $a,b\in L$ with $ab\leq \sqrt{q}$. Therefore $a^mb^m=(ab)^m\leq q$ for some positive integer $m$. If $a^mb^m\neq 0$, then there exists $s_1\in S$ such that $s_1a^m\leq q$ or $s_1b^m\leq \sqrt{q}$. Therefore $(s_1a)^m\leq s_1a^m\leq q$ or $(s_1b)^m=s_1b^m\leq \sqrt{q}$. This implies $s_1a\leq \sqrt{q}$ or $s_1b\leq \sqrt{q}$. Assume that $a^mb^m= 0$. Since $0$ is an $S$-primary element of $L$, there exists $s_2\in S$ such that $s_2a^m= 0$ or $s_2b^m\leq \sqrt{0}$. Therefore $s_2a\leq \sqrt{0}\leq \sqrt{q}$ or $s_2b\leq \sqrt{0}\leq \sqrt{q}$. Choose $s=s_1s_2$. Then for all $c,d\in L$ with $cd\leq \sqrt{q}$, $sc\leq \sqrt{q}$ or $sd\leq \sqrt{q}$. Thus, $\sqrt{q}$ is an $S$-prime element of $L$.
\end{proof}

\begin{prop}
	Let $S$ be a multiplicatively closed subset of a $c$-lattice $L$ and $q$ is a  weakly $S$-primary element of $L$. If $p$ is any element of $L$ such that   $t\leq p$ for some $t\in S$, then $pq$ and  $p\wedge q$ are  weakly $S$-primary elements of $L$.
\end{prop}
\begin{proof} Since $t \nleq q$ for all $t \in S$, we have  $t\nleq p\wedge q$  and $t \nleq pq$ for all $t\in S$. let $c,d\in L$ with $0 \neq cd \leq p q$. Then $ 0 \neq cd\leq q$. Since $q$ is a  weakly $S$-primary element of $L$, there exists $s\in S$ such that $sc\leq q$ or $sd\leq \sqrt{q}$. Take $t'\in S$ such that $t'\leq p$. Then $t'sc\leq pq $ or $t'sd \leq \sqrt{q}p \leq  \sqrt{pq}$. Hence, $p q$ is a weakly $S$-primary element of $L$.
	
Now, let $a,b\in L$ with $0 \neq ab\leq p\wedge q$. Then $ 0 \neq ab\leq q$. Since $q$ is a  weakly $S$-primary element of $L$, there exists $s'\in S$ such that $s'a\leq q$ or $s'b\leq \sqrt{q}$. Take $t^{''}\in S$ such that $t^{''}\leq p$. Then $t^{''}s'a\leq pq \leq  p\wedge q$ or $t^{''}s'b\leq \sqrt{q}p \leq  \sqrt{q}\wedge p \leq \sqrt{q}\wedge \sqrt{p}=\sqrt{p\wedge q}$. Hence, $p\wedge q$ is a weakly $S$-primary element of $L$. 
\end{proof}

\begin{prop}
	Let $S$ be a multiplicatively closed subset of a $c$-lattice $L$ and $q_1,q_2,\dots,q_n$ be   weakly $S$-primary elements of $L$. If $\sqrt{q_i}=\sqrt{q_j}$ for all $i,j=1,2,\dots,n$, then $\displaystyle q=\bigwedge_{i=1}^{n}q_i$ is a weakly $S$-primary element of $L$. In particular, if $q_1,q_2,\dots,q_n$ are not an $S$-primary element of $L$, then  $q$ is a weakly $S$-primary element of $L$. 
\end{prop}
\begin{proof}
	Since $q_1,q_2,\dots,q_n$ are  weakly $S$-primary elements of $L$, $t\nleq q_i$ for all $t\in S$; $i=1,2,\dots,n$. Hence $\displaystyle t\nleq \bigwedge_{i=1}^{n}q_i=q$ for all $t\in S$. Also for each $i$, there exists $s_i\in S$ such that for all $c,d\in L$ with $ (cd \neq 0) \leq q_i$, $s_ic\leq q_i$ or $s_id\leq \sqrt{q_i}$. Choose $\displaystyle  s=\prod_{i=1}^{n}s_i$. Let $a,b\in L$ with $ (ab \neq 0)\leq q$. Assume that $sa\nleq q$. Then $sa\nleq q_j$ for some $j$. If $s_ja\leq q_j$, then $\displaystyle sa=\bigg(\prod_{i=1}^{n}s_i\bigg)a\leq s_ja\leq q$, a contradiction. Hence $s_ja\nleq q_j$. Therefore $s_jb\leq \sqrt{q_j}$. Since $\sqrt{q_i}=\sqrt{q_j}$ for all $i,j=1,2,\dots,n$, $\displaystyle \sqrt{q}=\sqrt{\bigwedge_{i=1}^{n}q_i}=\bigwedge_{i=1}^{n}\sqrt{q_i}=\sqrt{q_i}$ for each $i=1,2,\dots,n$. Hence $s_jb\leq \sqrt{q}$. This implies $sb\leq s_jb\leq \sqrt{q}$. Thus, $q$ is a weakly $S$-primary element of $L$. In particular, if $q_1,q_2,\dots,q_n$ are not an $S$-primary element of $L$, then by (\ref{P5.7.2}) of Proposition \ref{P5.7}, $\sqrt{q_i}=\sqrt{0}$ for each $i=1,2,\dots, n$. Hence $\sqrt{q_i}=\sqrt{q_j}$ for all $i,j=1,2,\dots,n$. Thus, $q$ is a weakly $S$-primary element of $L$.
\end{proof}

\begin{prop}
	Let $S$ be a multiplicatively closed subset of a $c$-lattice $L$ and $i,~q$ are two elements of $L$ such that $i \in L_*$ and $i \leq q$.
	
	\begin{enumerate}
		\item If $q$ is a weakly $S$-primary element of $L$, then $q$ is a weakly $\bar{S}$-primary element of $L/i$,  where $\bar{S} = \{s \vee i \;|\; s\in S\}$.
		
		\item If $q$ is a weakly $\bar{S}$-primary element of $L/i$ and $i$ is an $S$-primary element of $L$, then $q$ is an $S$-primary element of $L$.
		
		\item If $q$ is a weakly $\bar{S}$-primary element of $L/i$ and $i$ is a weakly  $S$-primary element of $L$, then $q$ is a weakly   $S$-primary element of $L$.
	\end{enumerate}
\end{prop}

\begin{proof}
	\begin{enumerate}
		\item The proof is analogous to the proof of the Proposition \ref{P3.6}.
		
		\item Suppose that  $q$ is a weakly $\bar{S}$-primary element of $L/i$ and $i$ is an $S$-primary element of $L$. Since $q$ is a weakly $\bar{S}$-primary element of $L/i$, $t \nleq q$ for all $t \in S$.  Let $c,~d \in L$ such that $cd \leq q$. Since $i \leq q$, we get $(c \vee i ) (d \vee i) = cd \vee ci \vee di \vee i^2  = cd \vee i \leq q$. There are two cases. 
		
		First case is $cd \nleq i$ and second case is $cd \leq i$. Consider the first case $cd \nleq i$. This gives $cd \vee i \neq i$ in $L/i$.  As $(c \vee i), (d \vee i) \in L/i$ and $q$ is an $\bar{S}$-primary element of $L/i$.  there exists $s \vee i \in \bar{S}$ such that $(s \vee i) (c \vee i) \leq q$ or $(s \vee i) (d \vee i) \leq \sqrt{q}$. This gives $sc \vee si \vee ci \vee i^2 \leq q$ or $sd \vee si \vee di \vee i^2 \leq \sqrt{q}$. As $i \leq q \leq \sqrt{q}$, we get $ sc \vee si \vee ci \vee i^2 \vee i = sc \vee i  \leq q$ or $ sd \vee si \vee di \vee i^2 \vee i  = sd \vee i \leq \sqrt{q} $. Therefore $sc \leq sc \vee i \leq q $ or $sd \leq sd \vee i \leq \sqrt{q}$. Now, Consider the  second case $cd \leq i$. This gives $cd \vee i = i$ in $L/i$. Therefore $(c \vee i ) (d \vee i) = cd \vee ci \vee di \vee i^2  = cd \vee i \leq i$. Since $i$ is an $S$-primary element of $L$, there exists $s' \in S$ such that $s' (c \vee i) \leq i \leq q$ or $ s' (d \vee i) \leq \sqrt{i} \leq \sqrt{q}$. Hence in either case $q$ is an $S$-primary element of $L$.
		
		\item The proof is analogous to the proof of $(2)$.  
\end{enumerate}\end{proof}

\begin{rem}
	If $q_1$ is a weakly $S_1$-primary element of $c$-lattice $L_1$ and $q_2$ is a weakly $S_2$-primary element of $c$-lattice $L_2$, then $(q_1,q_2)$ need not be a weakly $(S_1\times S_2)$-primary element of $L_1\times L_2$. For this, consider $L_1=L_2=Id(\mathbb{Z})$ and $S_1=S_2=\{\mathbb{Z}\}$. Then $q_1=9\mathbb{Z}$ is weakly $S_1$-primary element of $L_1$ and $q_2=4\mathbb{Z}$ is weakly $S_2$-primary element of $L_2$. However, $(q_1,q_2)=(9\mathbb{Z},4\mathbb{Z})$ is not a weakly $(S_1\times S_2)$-primary element of $L_1\times L_2$ as $(0,0)\neq (9\mathbb{Z},\mathbb{Z})(\mathbb{Z},4\mathbb{Z})=(9\mathbb{Z},4\mathbb{Z})\leq (9\mathbb{Z},4\mathbb{Z})=(q_1,q_2)$, but $(9\mathbb{Z},\mathbb{Z})\nleq (q_1,q_2)$ and $(\mathbb{Z},4\mathbb{Z})\nleq (3\mathbb{Z},2\mathbb{Z})=\sqrt{(q_1,q_2)}$. Thus, $(q_1,q_2)$ is not a weakly $(S_1\times S_2)$-primary element of $L_1\times L_2$.
\end{rem}

\begin{thm}
	Let $S_1$, $S_2$ be the multiplicatively closed subsets of c-lattices $L_1$, $L_2$ respectively and $q_1\in L_1$, $q_2\in L_2$ be the non-zero elements. Let $L=L_1\times L_2$, $S=S_1\times S_2$ and $q=(q_1,q_2)$. Then the following are equivalent.
	\begin{enumerate}
		\item\label{T5.21.1} $q$ is a weakly $S$-primary element of $L$.
		\item\label{T5.21.2} $q_1$ is an $S_1$-primary element of $L_1$ and $t_2\leq q_2$ for some $t_2\in S_2$ or $q_2$ is an $S_2$-primary element of $L_2$ and $t_1\leq q_1$ for some $t_1\in S_1$.
		\item\label{T5.21.3} $q$ is an $S$-primary element of $L$.
	\end{enumerate}
\end{thm}
\begin{proof}
	(\ref{T5.21.1}) $\implies$ (\ref{T5.21.2}) Assume on the contrary that $t_1\nleq q_1$ for all $t_1\in S_1$ and $t_2\nleq q_2$ for all $t_2\in S_2$. Since $q_1$ and $q_2$ are non-zero, choose a non-zero element $(c,d)\leq q$. Then $(0,0)\neq (c,1)(1,d)=(c,d)\leq q$. Since $q$ is a weakly $S$-primary element of $L$, there exists $s=(s_1,s_2)\in S$ such that $s(c,1)\leq q$ or $s(1,d)\leq \sqrt{q}$. If $s(c,1)\leq q$, then $s_1c\leq q_1$ and $s_2\leq q_2$, a contradiction to our assumption. If $s(1,d)\leq \sqrt{q}=(\sqrt{q_1},\sqrt{q_2})$, then $s_1\leq \sqrt{q_1}$ and $s_2d \leq \sqrt{q_2}$. This implies $s_1^n\leq q_1$ for some positive integer $n$, a contradiction to our assumption. Hence, $t_1\leq q_1$ for some $t_1\in S_1$ or $t_2\leq q_2$ for some $t_2\in S_2$. Without loss of generality, assume that $t_1\leq q_1$ for some $t_1\in S_1$. Since $q$ is a weakly $S$-primary element, $t\nleq q$ for all $t\in S$. Therefore $(t_1,t_2)\nleq (q_1,q_2)$ for all $t_2\in S_2$. As $t_1\leq q_1$, we have $t_2\nleq q_2$ for all $t_2\in S_2$. Suppose $a,b \in L_2$ with $ab \leq q_2$. Choose $0\neq u\in S_1$ such that $u\leq q_1$. Then $(0,0)\neq (u,a)(1,b)\leq (q_1,q_2)=q$. This implies $s(u,a)=(s_1u,s_2a)\leq q$ or $s(1,b)=(s_1,s_2b)\leq \sqrt{q}$. Hence $s_2a\leq q_2$ or $s_2b \leq \sqrt{q_2}$. Thus, $q_2$ is an $S_2$-primary element of $L_2$.
	
	(\ref{T5.21.2}) $\implies$ (\ref{T5.21.3}) Assume that $q_1$ is an $S_1$-primary element of $L_1$ and $t_2\leq q_2$ for some $t_2\in S_2$. Let $c_1,d_1\in L_1$ and $c_2,d_2\in L_2$ be such that $(c_1,c_2)(d_1,d_2)\leq q$. Then $c_1d_1\leq q_1$. Since $q_1$ is an  $S_1$-primary element, there exists $t_1\in S_1$ such that $t_1c_1\leq q_1$ or $t_1d_1\leq \sqrt{q_1}$. Let $t=(t_1,t_2)$. Then $t(c_1,c_2)=(t_1c_1,t_2c_2)\leq (q_1,q_2)=q$ or $t(d_1,d_2)=(t_1d_1,t_2d_2)\leq (\sqrt{q_1},\sqrt{q_2})=\sqrt{q}$. Thus, $q$ is an $S$-primary element of $L$.
	
	(\ref{T5.21.3}) $\implies$ (\ref{T5.21.1}) Obvious.
\end{proof}

\begin{cor}
	Let $S_1, S_2, \dots , S_n$ be the multiplicatively closed subsets of $c$-lattices $L_1, L_2, \dots , L_n$ respectively and $q=(q_1,q_2,\dots,q_n)\in L$. Let $L=L_1\times L_2\times \dots \times L_n$ and $S=S_1\times S_2\times \dots \times S_n$. Then the following are equivalent.
	\begin{enumerate}
		\item
		$q$ is a weakly $S$-primary element of $L$.
		\item
		For some $j\in \{1,2,\dots,n\}$, $q_j$ is an $S_j$-primary element of $L_j$ and for all $i\in \{1,2,\dots,n\}\setminus \{j\}$, there exists $t_i\in S_i$ such that $t_i\leq q_i$.
	\end{enumerate}
\end{cor}

\end{document}